\documentclass[10pt]{amsart}

\usepackage{amssymb,amsmath,amsthm,amsfonts,microtype}
\usepackage{graphicx}
\usepackage{pdfsync}

\usepackage{tikz,tikz-3dplot}
\usetikzlibrary{shapes,positioning,intersections,quotes}
\usepackage[font=small,labelfont=bf]{caption}

\newcommand{\comment}[1]{}

\newtheorem{lem}{Lemma}
\newtheorem{propn}{Proposition}
\newtheorem{cor}{Corollary}
\newtheorem{thm}{Theorem}

\newtheorem*{thm0}{Theorem 0}

\theoremstyle{remark}

\theoremstyle{definition}

\newcommand{\R}{\mathbb R}

\newcommand{\Z}{\mathbb Z}

\newcommand{\N}{\mathbb N}

\newcommand{\F}{\mathcal F}
\newcommand{\FF}{\mathcal{F}}

\newcommand{\Tr}{\Delta}
\newcommand{\de}{\delta}

\newcommand{\si}{\sigma}

\newcommand{\om}{\omega}
\newcommand{\la}{\lambda}
\newcommand{\lm}{\lambda}

\newcommand{\subs}{\subseteq}

\newcommand{\uy}{\underline{y}}
\newcommand{\uz}{\underline{z}}

\newcommand{\be}{\begin{equation}}
\newcommand{\ee}{\end{equation}}
\newcommand{\eq}{\begin{equation}}
\newcommand{\bee}{\begin{equation*}}
\newcommand{\eee}{\end{equation*}}
\newcommand{\ls}{\lesssim}

\begin{document}
\title{Multilinear maximal operators associated to simplices}
\author{Brian Cook\quad\quad Neil Lyall \quad\quad \'{A}kos Magyar}
\thanks{The second and third authors were partially supported by grants NSF-DMS 1702411 and NSF-DMS 1600840, respectively.}

\address{Department of Mathematics,Virginia Tech, Blacksburg, VA 24061, USA}
\email{briancookmath@gmail.com}
\address{Department of Mathematics, The University of Georgia, Athens, GA 30602, USA}
\email{lyall@math.uga.edu}
\address{Department of Mathematics, The University of Georgia, Athens, GA 30602, USA}
\email{magyar@math.uga.edu}


\begin{abstract} 
We establish $L^{p_1}\times\cdots\times L^{p_k}\to L^r$ and $\ell^{p_1}\times\cdots\times \ell^{p_k}\to \ell^r$ type bounds for multilinear maximal operators associated to averages over isometric copies of a given non-degenerate $k$-simplex in both the continuous and discrete settings. These provide natural extensions of $L^p\to L^p$ and $\ell^p\to \ell^p$ bounds for Stein's spherical maximal operator and the discrete spherical maximal operator, with each of these results serving as a key ingredient of the respective proofs.
 \end{abstract}
\maketitle

\setlength{\parskip}{3.5pt}

\section{Introduction} 

\subsection{The spherical maximal operator}
 Let $d\geq 3$ and $\la>0$. For $f:\R^d\to\R$ define the averages
\bee
\mathcal{A}_\la f(x)= \int_{S^{d-1}} f(x+\la y)\,d\si(y)
\eee
and the maximal operator
\bee
\mathcal{A}_\ast f(x)= \sup_{\la>0} |A_\la f(x)|\eee
where $\si$ denotes the normalized surface area measure on the unit sphere $S^{d-1}=\{x\in\R^d:\ |x|=1\}$.

Stein's spherical maximal function theorem \cite{St}, states that for $p>d/(d-1)$ one has the estimate
\be\label{2.3'}
\|\mathcal{A}_\ast f\|_p \leq C_{p,d}\,\|f\|_p
\ee
where by $\|f\|_p$ denotes the $L^p(\R^d)$ norm of the function $f$. Note that Bourgain \cite{B} extended the above result for $d=2$, and that the condition  $p>d/(d-1)$ is sharp.

\subsection{The discrete spherical maximal operator}

The study of discrete analogues of central constructs of Euclidean harmonic analysis, initiated by Bourgain \cite{B1,B2,B3}, has grown into a vast, active area of research. An important result in this development is the $\ell^p$-boundedness of the so-called discrete spherical maximal operator \cite{MSW},
we now recall this operator and the main result of \cite{MSW}. 

Let $d\geq 5,\,\la^2\in\N$, and $N_\la :=|\{y\in\Z^d:\ |y|=\la\}|.$ It is well-known, see for example \cite{V}, that \[c_d\la^{d-2}\leq N_\la \leq C_d\la^{d-2}\] for some constants $0<c_d<C_d$. 
For $f:\Z^d\to\R$ define the averages
\bee
A_\la f(x)= N_\la^{-1} \sum_{|y|=\la} f(x+y)
\eee
and the maximal operator
\bee
A_\ast f(x)= \sup_{\la} |A_\la f(x)|.\eee

The variables $x,y$ in the two equations above, and throughout this short note whenever we are considering discrete operators, are always assumed to in $\Z^d$, unless explicitly specified otherwise. Furthermore, in the discrete setting the parameter $\la$ will always be assumed be in $\sqrt{\N}$, that is satisfy $\la^2\in\N$. 

In \cite{MSW} it was shown that for $p>d/(d-2)$ one has the estimate
\eq\label{2.3}
\|A_\ast f\|_p \leq C_{p,d}\,\|f\|_p
\ee
where $\|f\|_p$ 
denotes the $\ell^p(\Z^d)$ norm of the function $f$.
It was further noted in \cite{MSW} that the condition that $d\geq 5$ and $p>d/(d-2)$ are both sharp.

\section{Multilinear maximal operators associated to simplices}

The aim of this short note is to show that estimates (\ref{2.3'}) and (\ref{2.3})  imply $L^{p_1}\times\cdots\times L^{p_k}\to L^r$ and $\ell^{p_1}\times\cdots\times \ell^{p_k}\to \ell^r$ type bounds for certain, seemingly more singular, multilinear maximal operators associated to averages over similar copies of a given non-degenerate simplex in the continuous and discrete settings, respectively.

\subsection{Multilinear maximal operators associated to simplices in $\mathbb{R}^d$}\

Let $k\in\N$ and let $\Tr=\{v_0=0,v_1,\ldots,v_k\}\subs\R^d$ be a non-degenerate $k$-simplex, that is assume that the vectors $v_1,\ldots,v_k$ are linearly independent. Given $\la>0$ we say that a simplex $\Tr'=\{y_0=0,y_1,\ldots,y_k\}\subs \R^d$ is \emph{isometric} to $\Tr$ if $|y_i-y_j|=\la |v_i-v_j|$ for all $0\leq i,j\leq k$. 
We will write $\Tr'\simeq \la\Tr$ in this case. 

For a family of functions $f_1,\dots,f_k:\R^d\to\R$ with $d\geq k+1$ and $\la>0$ we define the multilinear averages
\be
\mathcal{A}_{\lm}(f_1,\dots,f_{k})(x)= \int_{SO(d)} f_1(x+\lm \cdot U(v_1))\cdots f_k(x+\lm\cdot U(v_k))\,d\mu(U)
\ee
where $\mu$ denotes the Haar measure on $SO(d)$ and the associated maximal operator
\eq
\mathcal{A}_\ast(f_1,\ldots,f_k)(x) = \sup_{\la>0} |\mathcal{A}_\la(f_1,\ldots,f_k)(x)|.
\ee

We quickly record the following trivial observation, which essentially appears in both \cite{HLM} and \cite{PS}.
\begin{thm0}
Let $k\in\N$ and  $\Tr=\{v_0=0,v_1,\ldots,v_k\}\subs\R^d$ be a non-degenerate $k$-simplex.

If $d\geq k+1$ and $1/r=1/p_1+\cdots+1/p_k<(d-1)/d$, then
\bee
\|\mathcal{A}_*(f_1,\dots,f_k)\|_r\leq C_{d,k,\Delta}\,\|f_1\|_{p_1}\cdots\|f_k\|_{p_k.}
\eee
\end{thm0}
\begin{proof}
For each $\lm>0$ we have 
\[
|\mathcal{A}_{\lm}(f_1,\dots,f_{k})(x)|\leq \|f_2\|_\infty\cdots\|f_k\|_\infty\,\int |f_{1}(x+\lm y_1)|\,d\sigma(y_{1})
\]
where $\sigma$ denotes the normalized measure on the sphere $S^{d-1}(0,|v_1|)=\{y\in\R^d\,:\, |y|=|v_1|\}$. It therefore follows from (\ref{2.3'}) that $\mathcal{A}_*$ is bounded on $L^p\times L^\infty\times\cdots\times L^\infty\to L^p$, whenever $p>d/(d-1)$ and $d\geq k+1$.

Theorem 0 now follows by symmetry and interpolation.
\end{proof}

It is straightforward to verify, following the ideas in Section 6 of \cite{PS}, that a necessary condition for $\mathcal{A}_*$ to be bounded on $L^{p_1}\times \cdots\times L^{p_k}\to L^r$ is that $1/r=1/p_1+\cdots+1/p_k$ with  $p_1,\dots,p_k>d/(d-1)$. It is therefore of interest to obtain estimates outside of the ``trivial region" given by Theorem 0, namely for $1/p_1+\cdots+1/p_k\geq (d-1)/d$. Theorems \ref{Thm3} and \ref{Thm3'}, in Section \ref{cont}, establishes precisely this in sufficiently high dimensions. We establish a convex region of points $(1/p_1,\ldots,1/p_k)$ for which $\mathcal{A}_*$ is bounded that in particular contains the cube $q^{-(k-1)}\cdot (d-1)/d\cdot[0,1)^k$ with $q=m/m-1$, in dimensions $d>2km+2$. In particular, when $k=2$, we establish that $\mathcal{A}_*$ is bounded on $L^{p_1}\times L^{p_2}\to L^r$, whenever $1/r=1/p_1+1/p_2$ with $p_1,p_2>m/(m-1)\cdot d/(d-1)$ and $d\geq 2m$, see Figure \ref{fig1}.

\subsection{Multilinear maximal operators associated to simplices in $\mathbb{Z}^d$}\

Let $k\in\N$ and let $\Tr=\{v_0=0,v_1,\ldots,v_k\}\subs\Z^d$ be a non-degenerate $k$-simplex. Given $\la\in\sqrt{\N}$ we say that a simplex $\Tr'=\{y_0=0,y_1,\ldots,y_k\}\subs \Z^d$ is \emph{isometric} to $\Tr$ if $|y_i-y_j|=\la |v_i-v_j|$ for all $0\leq i,j\leq k$. 
We will again write $\Tr'\simeq \la\Tr$ in this case and now denote by $N_{\la\Tr}$ the number of isometric copies of $\la\Tr$, namely 
\[N_{\la\Tr} := |\{(y_1,\ldots,y_k)\in \Z^{dk}:\ \Tr'=\{0,y_1,\ldots,y_k\}\simeq \la\Tr\}|.\]
Note that for $k=1$ and $v_1=(1,0,\ldots,0)$ we have that $N_{\la\Tr}=N_\la$.

Given any simplex $\Tr=\{v_0=0,v_1,\ldots,v_k\}\subseteq\R^d$, we introduce the associated \emph{inner product matrix} $T=T_\Tr= (t_{ij})_{1\leq i,j\leq k}$ with entries $t_{ij}:= v_i\cdot v_j$, where ``$\cdot$"
stands for the dot product in $\R^d$. Note that $T$ is a positive semi-definite matrix with integer entries and $T$ is positive definite if and only if $\Tr$ is non-degenerate. 
It is easy to see that $\Tr'\simeq \la\Tr$, with $\Tr'=\{y_0=0,y_1,\ldots,y_k\}$, if and only if 
\eq\label{2.4}
y_i\cdot y_j=\la^2 t_{ij}\quad \text{for all}\quad 1\leq i,j\leq k.
\ee

Extending the work of Siegel \cite{S} and Raghavan \cite{R}, Kitaoke \cite{K} has proved that if $\Tr$ is non-degenerate, then one has the estimate
\eq\label{2.5}
c_{d,k}\det\,(\la^2 T)^{(d-k-1)/2} \leq N_{\la\Tr} \leq C_{d,k}\det\,(\la^2 T)^{(d-k-1)/2}
\ee
in dimensions $d\geq2k+3$ for $\la\geq \la_{d,k,\Tr}$. It is important to note that the constants $0<c_{d,k}<C_{d,k}$ depending only on the parameters $d$ and $k$ and are independent of the matrix $T$ and hence the simplex $\Tr$. For a self contained treatment of the upper bound in \eqref{2.5}, see Lemma 2.2 in \cite{M}. In particular for sufficiently large $\la$ one has that $N_{\la\Tr}>0$, in fact $N_{\la\Tr}\asymp \la^{kd-k(k+1)}$ with implicit constants may depending on $\Tr$. 

\smallskip

For a family of functions $f_1,\dots,f_k:\Z^d\to\R$ and $\la\in\sqrt{\N}$ such that $N_{\la\Tr}>0$ we define the multilinear averages
\eq\label{2.6}
A_\la(f_1,\ldots,f_k)(x) = N_{\la\Tr}^{-1} \sum_{y_1,\ldots,y_k} f_1(x+y_1)\cdots f_k(x+y_k)\,S_{\la^2 T} (y_1,\ldots,y_k)
\ee
where $S_{\la^2 T}(y_1,\ldots,y_k)=1$ if $y_1,\ldots,y_k\in\Z^d$ satisfies \eqref{2.4} and is equal to 0 otherwise, i.e. the indicator function of the relation $\Tr'\simeq \la\Tr$,
and the associated maximal operator
\be\label{2.7}
A_\ast(f_1,\ldots,f_k)(x) = \sup_{\la} |A_\la(f_1,\ldots,f_k)(x)|
\ee
where the supremum is restricted to those $\la\in\sqrt{\N}$ for which $N_{\la\Tr}>0$.

We remark that there is no direct analogue of Theorem 0 in the discrete setting. This difficulty arises from the fact that for fixed $y_1$, we do not necesssailly have control over the count
\[\sum_{y_2,\ldots,y_{k}} S_{\la^2 T} (y_1,\ldots,y_k).
\]
Our results do however  rely on leveraging the fact that the above sum is well behaved on average.

\section{Main results for our Discrete Multilinear Maximal Operators}

In the discrete setting  we choose to present our results in an increasing order of generality, first presenting the following special case of our most general result in the special case of discrete bilinear maximal operators associated to triangles.

\begin{thm}\label{Thm0} Let $\Tr=\{v_0=0,v_1,v_2\}\subs\Z^d$ be a non-degenerate triangle. 

\begin{itemize}
\item[(i)] If $d\geq9$, $r>2d/(d-2)$, and $1\leq p_1,p_2\leq\infty$ with $1/r=1/p_1+1/p_2$, then one has the estimate
\be\label{2.8}
\|A_\ast(f_1,f_2)\|_r \leq C_{d,\Tr}\, \|f_1\|_{p_1}\|f_2\|_{p_2}.
\ee
\smallskip
\item[(ii)]
If $d\geq11$, then for any $r>d/(d-2)$ and $p_1,p_2>2d/(d-2)$ that satisfies $1/r=1/p_1+1/p_2$,
one has
\bee
\|A_\ast(f_1,f_2)\|_r \leq C_{d,\Tr}\, \|f_1\|_{p_1} \|f_2\|_{p_2}.
\eee
\end{itemize}
\end{thm}

For a visualization of those $p_1$ and $p_2$ for which Theorem \ref{Thm0} gives boundedness for these discrete bilinear maximal operators, see Figure \ref{fig3} (with $m=2$).

Note that if we know that  $A_\ast$ is bounded on $\ell^{p_{1}}\times\ell^{p_{2}}\to\ell^{r}$, then we automatically get all bounds 
$\ell^{q_{1}}\times\ell^{q_{2}}\to\ell^{s}$ for all $q_1\leq p_{1}$, all $q_2\leq p_{2}$, and $s\geq r$ due to the nested properties of the discrete norms.

Furthermore, note that in Theorem \ref{Thm0} above, and in all subsequent theorem and propositions in this paper (except for Theorem \ref{Thm2}), part (ii) implies part (i) for the range of dimensions in which part (ii) holds.

\smallskip

We remark that it was independently and simultaneously established by Anderson, Kumchev and Palsson in \cite{AKP} that in dimensions $d\geq 9$, with $\Tr$ being a equilateral triangle, that estimate \eqref{2.8} holds in the larger range $r>\max\{32/(d+8),(d+4)/(d-2)\}$. 
Their result follows as a direct corollary of $\ell^p\times\ell^\infty\to\ell^p$ bounds obtained by employing very different methods than those contained in this short note. 

Our proof of (i) above also follows from  $\ell^p\times\ell^\infty\to\ell^p$ estimates. In  Section \ref{last} we discuss a generalization of our method that allows us to obtain better  bounds in larger dimensions. In particular, we  obtain $\ell^{p_{1}}\times\ell^{p_{2}}\to\ell^{r}$ bounds whenever $r>m/(m-1)\,\cdot\,d/(d-2)$ and $1\leq p_1,p_2\leq\infty$ with $1/r\leq1/p_1+1/p_2$, provided $d\geq2m+5$. This represents an improvement on the results in \cite{AKP} for $d\geq15$. See Theorem \ref{Thm2}, with $k=2$, and Figure \ref{fig3}.

We remark that our proof of (ii) above, which we emphasize
gives non-trivial estimates for a range of $p_1$ and $p_2$ for any given $r>d/(d-2)$, provided $d\geq11$, does not follow as a corollary of $\ell^p\times\ell^\infty\to\ell^p$ estimates.

\medskip

Before stating our next result, Theorem \ref{Thm1} below, which generalizes Theorem \ref{Thm0} to multilinear maximal operators associated to $k$-simplices, we define for each integer $k\geq 2$, a symmetric convex region $\mathcal{C}_k\subseteq[0,1]^k$. We define $\mathcal{C}_k$ to be all those points $(x_1,\dots,x_k)\in[0,1]^k$ with $x_1+\cdots+x_k<1$ that also have the property that for any $1\leq j\leq k-1$ one has $y_1+\cdots+y_j<1-2^{-j}$ for any choice  $\{y_1,\dots,y_j\}\subset\{x_1,\dots,x_k\}$. 

We note, in particular, that if $(x_1,\dots,x_k)\in \mathcal{C}_k$, then $0\leq x_1,\dots,x_k<1/2$, and that both the points $(1/k,\dots,1/k)$ and $(1/2,0,\dots,0)$, while not in $\mathcal{C}_k$, are contained in the boundary of $\mathcal{C}_k$. See Figure \ref{fig5} below.

\begin{figure}[h]
\tdplotsetmaincoords{80}{110}
\begin{tikzpicture}[tdplot_main_coords]
\draw[->]  (0,0,0) -- (7.5,0,0) node[anchor=north]{$x_1$};
\draw
[->] (0,0,0) -- (0,4.5,0) node[anchor=west]{$x_2$};
\draw[->] (0,0,0) -- (0,0,4.5) node[anchor=east]{$x_3$};

\draw[fill=gray!15, fill opacity=0.8, style = dotted] (3,0,0) -- (3,1,0)  -- (3,1,1) -- (3,0,1) -- cycle;
\draw[fill=gray!20, fill opacity=0.8, style = dotted]  (3,1,0) -- (3,1,1) -- (1.8,1.7,1) --  (1.8,1.7,0) -- cycle;
\draw[fill=gray!25, fill opacity=0.8, style = dotted] (1.8,1.7,1)  --  (1.8,1.7,0) -- (0,1.7,0) -- (0,1.7,1)  -- cycle;
\draw[fill=gray!10, fill opacity=0.8, style = dotted] (3,1,1)  -- (1.7,1,1.8) -- (1.7,0,1.8) -- (3,0,1) -- cycle;
\draw[fill=gray!15, fill opacity=0.8, style = dotted] (1.8,1.7,1) -- (3,1,1) -- (1.7,1,1.8)-- cycle;
\draw[fill=gray!15, fill opacity=0.8, style = dotted] (1.7,1,1.8) -- (1.8,1.7,1) -- (0,1.7,1) -- (0,1,1.8) -- cycle;

\draw[fill=gray!10, fill opacity=0.8, style = dotted] (0,0,1.8) -- (0,1,1.8) -- (1.7,1,1.8) --(1.7,0,1.8) -- cycle;

\node [draw, circle, minimum width=4pt, inner sep=0pt] at (6,0,0) {};
\node [draw, color=red,circle, minimum width=4pt, inner sep=0pt] at (3,0,0) {};

\node[above left =2pt of {(6,0,0)}, outer sep=2pt] {\small$(1,0,0)$};
\node [draw, color=red, circle, minimum width=4pt, inner sep=0pt] at (3,1,0) {};
\node [draw, color=red, color=red, circle, minimum width=4pt, inner sep=0pt] at (3,1,1) {};
\node[color=red, below left=0pt of {(3.1,1.1,1.1)}, outer sep=0pt] {\small$\vec{x}_1$};

\node [draw, color=red, circle, minimum width=4pt, inner sep=0pt] at (1.8,1.7,0) {};

\node [draw, color=red, circle, minimum width=4pt, inner sep=0pt] at (3,0,1) {};

\node [draw, color=red, circle, minimum width=4pt, inner sep=0pt] at (1.8,1.7,1) {};
\node[color=red, below right=0pt of {(1.8,1.7,1.1)}, outer sep=0pt] {\small$\vec{x}_2$};

\node [draw, color=red, circle, minimum width=4pt, inner sep=0pt] at (1.7,0,1.8) {};

\node [draw,  color=red, circle, minimum width=4pt, inner sep=0pt] at (0,0,1.8){};

\node [draw, circle, minimum width=4pt, inner sep=0pt] at (0,0,3.75) {};
\node[above right=2pt of {(0,0,3.75)}, outer sep=2pt] {\small$(0,0,1)$};

\node [draw, color=red, circle, minimum width=4pt, inner sep=0pt] at (1.7,1,1.8) {};
\node[color=red, above=0pt of {(1.7,1,1.8)}, outer sep=0pt] {\small$\vec{x}_3$};

\node [draw, color=red, circle, minimum width=4pt, inner sep=0pt] at (0,1.7,1) {};
\node [draw, color=red, circle, minimum width=4pt, inner sep=0pt] at (0,1,1.8) {};

\node [draw,  circle, minimum width=4pt, inner sep=0pt] at (0,3.7,0) {};

\node [draw, color=red, circle, minimum width=4pt, inner sep=0pt] at (0,1.7,0) {};
\node[above right =2pt of {(0,3.7,0)}, outer sep=4pt] {\small$(0,1,0)$};

\draw[color=black, style=dashed] (6,0,0) -- (0,0,3.75) -- (0,3.7,0) -- cycle;

\comment{
\draw[color=black, style=dashed] (6,4,0) -- (6,4,3.9);
\draw[color=black, style=dashed] (6,0,0) -- (6,4,0) -- (0,3.7,0);
\draw[color=black, style=dashed] (6,0,0) -- (6,0,3.8) -- (0,0,3.75);
\draw[color=black, style=dashed] (0,0,3.75) -- (0,3.7,3.9) -- (0,3.7,0);
\draw[color=black, style=dashed] (6,0,0) -- (6,4,0) -- (0,3.7,0);
\draw[color=black, style=dashed] (6,0,3.8) -- (6,4,3.9) -- (0,3.7,3.9);
}

\end{tikzpicture}
\caption{Illustration of $\mathcal{C}_3$ where $\vec{x}_1=(1/2, 1/4, 1/4)$, $\vec{x}_2=( 1/4, 1/2,1/4)$, and $\vec{x}_3=( 1/4, 1/4,1/2)$.
}\label{fig5}
\end{figure}
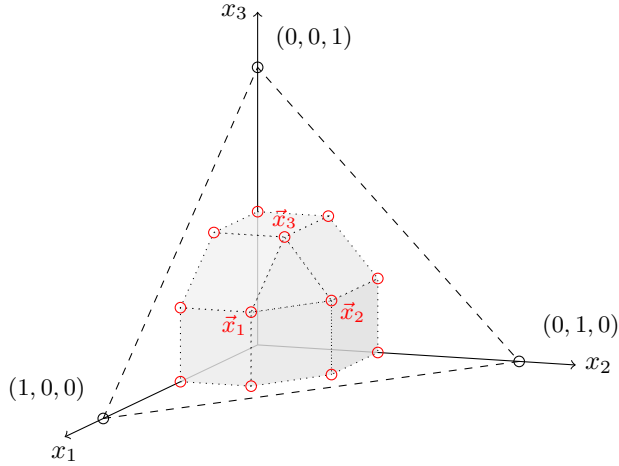

\smallskip

\begin{thm}\label{Thm1} Let $k\in\N$ and $\Tr=\{v_0=0,v_1,\ldots,v_k\}\subs\Z^d$ be a non-degenerate $k$-simplex. 


\begin{itemize}
\item[(i)] If $d\geq4k+1$, $r>2d/(d-2)$, and $1\leq p_1,\dots,p_k\leq\infty$ with $1/r=1/p_1+\cdots+1/p_k$, then
one has 
\bee
\|A_\ast(f_1,\ldots,f_k)\|_r \leq C_{d,k,\Tr}\, \|f_1\|_{p_1}\cdots \|f_k\|_{p_k}.
\eee
\smallskip
\item[(ii)]
If $d\geq4k+3$, then for any $r>d/(d-2)$ and $p_1,\dots,p_k>2d/(d-2)$ whose reciprocals \[(1/p_1,\dots,1/p_k)\in (d-2)/d\cdot\mathcal{C}_k\] and satisfy $1/r=1/p_1+\cdots+1/p_k$, 
one has the estimate
\bee
\|A_\ast(f_1,\ldots,f_k)\|_r \leq C_{d,k,\Tr}\, \|f_1\|_{p_1}\cdots \|f_k\|_{p_k}.
\eee
\end{itemize}
\end{thm}

For a visualization of those $p_1,\dots,p_k$ for which Theorem \ref{Thm1} gives boundedness for these discrete multilinear maximal operators, see Figure \ref{fig4} (with $m=2$).

Note, as above, that if we know that  $A_\ast$ is bounded on $\ell^{p_{1}}\times\cdots\times\ell^{p_{k}}\to\ell^{r}$, then it is automatically bounded on  
$\ell^{q_{1}}\times\cdots\times\ell^{q_{k}}\to\ell^{s}$ for all $q_1\leq p_{1},\dots,q_k\leq p_{k}$, and $s\geq r$. 

In Section \ref{last} we discuss a generalization of our method that allows us to obtain better $\ell^{p_{1}}\times\cdots\times\ell^{p_{k}}\to\ell^{r}$ bounds provided that $d$ is sufficiently large. In particular, we  obtain $\ell^{p_{1}}\times\cdots\times\ell^{p_{k}}\to\ell^{r}$ bounds whenever $r>m/(m-1)\,\cdot\,d/(d-2)$ and $1\leq p_1,\dots,p_k\leq\infty$ with $1/r\leq1/p_1+\cdots+1/p_k$, provided $d\geq2m(k-1)+5$, and more general estimates whenever $d\geq2mk+3$.
See Theorem \ref{Thm2} and Figure \ref{fig4}.

We conclude discrete matters in Section \ref{example} by demonstrating that $\ell^{p}\times\ell^\infty\times\cdots\times\ell^{\infty}\to\ell^{p}$ boundedness fails for every $p\leq d/(d-2)$ in dimensions $d\geq 2k+3$.


\section{Proof of Theorem \ref{Thm1}}

The crucial ingredient in our proof of Theorem \ref{Thm1} is pointwise estimates for  $A_\ast(f_1,\dots f_k)$ in terms of the spherical maximal operator applied to appropriate powers of the functions $f_j$, specifically

\begin{propn}\label{Propn1}
Let $k\in\N$ and $\Tr=\{v_0=0,v_1,\ldots,v_k\}\subs\Z^d$ be a non-degenerate $k$-simplex. 

\begin{itemize}
\item[(i)] If $d\geq4k+1$, then for any $f_1,\dots,f_k:\Z^d\to\R$, one has
\be\label{4.9-}
A_\ast(f_1,\ldots,f_k)(x)\leq C_{d,k,\Delta}\,\|f_1\|_\infty\cdots\|f_{k-1}\|_\infty \ A_\ast (f_k^2)(x)^{1/2}
\ee
uniformly for $x\in\Z^d$.
\vspace{10pt}
\item[(ii)]
If $d\geq4k+3$, then for any $f_1,\dots,f_k:\Z^d\to\R$, one has
\be\label{4.9}
A_\ast(f_1,\ldots,f_k)(x)\leq C_{d,k,\Delta}\,A_\ast(f_1^2,\ldots,f_{k-1}^2)(x)^{1/2}\,A_\ast (f_k^2)(x)^{1/2}
\ee
and hence
\be\label{4.9'}
A_\ast(f_1,\ldots,f_k)(x)\leq C_{d,k,\Delta}\,A_\ast(f_1^{2^{k-1}})(x)^{1/2^{k-1}}A_\ast(f_2^{2^{k-1}})(x)^{1/2^{k-1}}\prod_{j=3}^k A_\ast(f_j^{2^{k+1-j}})(x)^{1/2^{k+1-j}}
\ee
uniformly for $x\in\Z^d$.
\end{itemize}
\end{propn}

We prove Proposition \ref{Propn1} in Section \ref{Propn1proof} below. 
It is straightforward to see that Theorem \ref{Thm1} (i) follows immediately from (\ref{4.9-}) and \eqref{2.3}, indeed these estimates imply
\[\|A_\ast(f_1,\ldots,f_k)\|_{p_k}\leq C_{d,k,\Delta}\,\|f_1\|_\infty\cdots\|f_{k-1}\|_\infty  \|A_\ast (f_k^2)\|^{1/2}_{p_k/2}\leq C_{d,k,\Delta}\, \|f_1\|_\infty\cdots\|f_{k-1}\|_\infty  \|f_k\|_{p_k}\]
provided $p_k>2d/(d-2)$. By symmetry and interpolation we then obtain part (i) of Theorem \ref{Thm1}.

Assuming the validity (\ref{4.9'}) for now, we can also quickly establish Theorem \ref{Thm1} (ii). An application of H\"older gives that
\[\|A_\ast(f_1,\ldots,f_k)\|_r\leq C_{d,k,\Delta}\|A_\ast(f_1^{2^{k-1}})\|_{p_1/2^{k-1}}^{1/2^{k-1}}\|A_\ast(f_2^{2^{k-1}})\|_{p_2/2^{k-1}}^{1/2^{k-1}}\prod_{j=3}^k \|A_\ast(f_j^{2^{k+1-j}})\|_{p_j/2^{k+1-j}}^{1/2^{k+1-j}}\]
whenever $1/r=1/p_1+\cdots+1/p_k$. Now if
\bee
p_1,p_2>2^{k-1}\frac{d}{d-2} \ \ \ \text{and} \ \ \ p_j>2^{k+1-j}\frac{d}{d-2} \ \ \text{for} \ \  3\leq j\leq k\eee
then by \eqref{2.3} we obtain 
\[\|A_\ast(f_1,\ldots,f_k)\|_r\leq C_{d,k,\Delta}\|f_1\|_{p_1}\cdots\|f_k\|_{p_k}\]
with $1/r=1/p_1+\cdots+1/p_k<(d-2)/d$. Theorem \ref{Thm1} (ii) now follows by symmetry and interpolation. \qed

\section{Proof of Proposition \ref{Propn1}}\label{Propn1proof}

The key ingredient of the proof of this proposition is an upper bound on the $\ell^1$ norm of the function $S_T(y_1,\ldots,y_k)$ defined in \eqref{2.4} (when $\la=1$), proved in Lemma 2.2 in \cite{M}, namely if $T=(t_{ij})$ is a positive definite integral $k\times k$ matrix then for $d\geq2k+3$ one has
\eq\label{3.2'}
\sum_{y_1,\ldots,y_k\in\Z^d} S_T(y_1,\ldots,y_k) \leq C_{d,k}\,\Bigl(\det(T)^{(d-k-1)/2} + |T|^{(d-k)(k-1)/2}\Bigr)
\ee
with $|T|:=(\sum_{i,j}t_{ij}^2)^{1/2}$. 


Let $\Tr=\{v_0=0,v_1,\ldots,v_k\}$ be a non-degenerate $k$-simplex with inner product matrix $T=(t_{ij})$. 
Note that for $\la\leq \la_{d,k,\Tr}$ we have that $N_{\la\Tr}\leq C_{d,k,\Tr}$ thus by H\"{o}lder's and Minkowski's inequalities we have that $\|A_\la(f_1,\ldots,f_k)\|_r\leq  C_{d,k,\Tr}\,\|f_1\|_{p_1}\ldots \|f_k\|_{p_k}$, whenever $1/p_1+\cdots+1/p_k=1/r$. Thus the supremum in \eqref{2.7} can be restricted to sufficiently large $\la$. Then because of $N_{\la\Tr}\asymp \la^{k(d-k-1)}$ one may replace the factor  $N_{\la\Tr}^{-1}$ with $\la^{-k(d-k-1)}$ in formula \eqref{2.6} and assume without loss of generality that $\la\geq \la_{d,k,\Tr}$.

\smallskip

We choose to focus first on establishing part (ii) of Proposition \ref{Propn1}.

\begin{proof}[Proof of Propostion \ref{Propn1} (ii)]
For a solution $y_1,\ldots,y_k$ to the system of equations \eqref{2.4} we will write $\uy_1=(y_1,\ldots,y_{k-1})$ to group the first $k-1$ variables and $T_1$ for the corresponding inner product matrix, i.e. for the $k-1\times k-1$ minor of $T$.
  For given $x\in\Z^d$, by the Cauchy-Schwarz inequality, in dimensions $d>2k$ we have
\begin{align*}
A_\la(f_1,\ldots,f_k)(x)^2 &\leq\ \la^{-d(k-1)+k(k-1)}\ \sum_{\uy_1} S_{\la^2T_1}(\uy_1) f_1^2(x+y_1)\cdots f_{k-1}^2(x+y_{k-1}) \\
&\quad\quad\quad\quad\quad\quad\quad\times\ \la^{-d(k+1)+k^2+3k}\ \sum_{\uy_1}\,\Bigl(\sum_{y_k} f_k(x+y_k) 
S_{\la^2 T}(\uy_1,y_k)\Bigr)^2\nonumber\\
& \leq\ A_\ast (f_1^2,\ldots,f_{k-1}^2)(x)\ B_\la(f_k,f_k)(x)\nonumber
\end{align*}
where 
\bee
B_\la(f_k,f_k)(x) = \la^{-d(k+1)+k^2+3k} \sum_{y_k,y_k'} f_k(x+y_k) f_k(x+y_k') W_{\la^2 T} (y_k,y_k')
\eee
with a weight function
\eq\label{4.3}
W_{\la^2 T} (y_k,y_k') = \sum_{\uy_1} S_{\la^2 T}(\uy_1,y_k)\, S_{\la^2 T}(\uy_1,y_k').
\ee

By a slight abuse of notation  let $S_\lm(y)=1$ if $|y|^2=t_{kk}\lm^2$ and equal to $0$ otherwise. Then one may write
\bee
B_\la(f_k,f_k)(x) = \la^{-d(k+1)+k^2+3k} \sum_{y_k,y_k'} f_k(x+y_k) f_k(x+y_k')S_\lm(y_k)S_\lm(y_{k'})W_{\la^2 T} (y_k,y_k')
\eee
and an application of Cauchy-Schwarz gives
\bee
B_\la(f_k,f_k)(x)^2 \leq \Bigl(\la^{-d+2} \sum_{y} f^2_k(x+y)S_\lm(y)\Bigr)^2 \Bigl(\lm^{-2dk+2k^2+6k-4}\sum_{y_k,y_k'}W_{\la^2 T} (y_k,y_k')^2\Bigr).
\eee

Thus, in order to establish (\ref{4.9}) and complete the proof of the proposition, it suffices to show that
\bee
\sum_{y_k,y_k'} W_{\la^2 T} (y_k,y_k')^2  \leq C\,\la^{2dk -2k^2-6k+4}
\eee
with a constant $C=C_{d,k,T}>0$.
By \eqref{4.3}, we have that
\bee
\sum_{y_k,y_k'} W_{\la^2 T} (y_k,y_k')^2 = \sum_{\uy_1,\uy_1',y_k,y_k'} 
S_{\la^2 T}(\uy_1,y_k)S_{\la^2 T}(\uy_1',y_k)S_{\la^2 T}(\uy_1,y_k')S_{\la^2 T}(\uy_1',y_k').
\eee

The above expression is the number of solutions $y_1,\ldots,y_k,y_1',\ldots,y_k'\in\Z^d$ to the system of quadratic equations
\begin{align}\label{4.7}
y_i\cdot y_j &=y_i'\cdot y_j'=\la^2 t_{ij},\ \text{for}\ \  1\leq i,j\leq k-1\nonumber\\
y_i\cdot y_k& = y_i'\cdot y_k = y_i\cdot y_k' = y_i'\cdot y_k' = \la^2 t_{ik}, \ \text{for}\ \  1\leq i\leq k-1\\
\ y_k\cdot y_k &=y_k'\cdot y_k'=\la^2 t_{kk}\nonumber.
\end{align}

For any solution $y_1,\ldots,y_k,y_1',\ldots,y_k'$ of the system \eqref{4.7} introduce the parameters $(s_{ij})_{1\leq i,j\leq k-1}$ and $s_{kk}$ such that 
\eq\label{4.8}
y_i\cdot y_j' =\la^2 s_{ij}\ \text{for}\ \ 1\leq i,j\leq k-1\ \text{and}\ \ y_k\cdot y_k'=\la^2 s_{kk}.
\ee

We call the set of parameters $S=(s_{ij},s_{kk})_{1\leq i,j\leq k-1}$ \emph{admissible} if the system \eqref{4.7}-\eqref{4.8} have a solution. For any admissible set of parameters $S$ let $\la^2 T_S$ denote the $2k\times 2k$ inner product matrix of the system \eqref{4.7}-\eqref{4.8}, and note that $\la^2 T_S$ is a positive semi-definite integral matrix with entries $O_T(\la^2)$. 

We consider two cases.


\underline{Case 1:} Assume that the matrix $T_S$ is positive definite. Then in dimensions $d\geq4k+3$ one may apply estimate \eqref{3.2'} to the matrix $\la^2 T_S$ which shows that the number of solutions to the system \eqref{4.7}-\eqref{4.8} is bounded by $C\,\la^{2dk-2k(2k+1)}$. Since there at most $C\,\la^{2(k-1)^2+2}$ admissible sets $S$, such admissible sets contribute to at most $C\,\la^{2dk-2k^2-6k+4}$ solutions to the system \eqref{4.7}, for some constant $C=C_{d,k,T}>0$.

\underline{Case 2:} Assume $\det (T_S)=0$. Then the vectors $y_1,\ldots,y_k,y_1',\ldots,y_k'$ are linearly dependent. Let $M:= \text{span}\{y_1,\ldots,y_k,y_1',\ldots,y_k'\}\subs\R^d$. Since $y_1,\ldots,y_k$ are linearly independent one may extend these vectors with vectors $y_{i_1}',\ldots y_{i_l}'$, for some $1\leq l<k$, to obtain a basis of of the vector space $M$. Write $I=\{i_1,\ldots,i_l\}$, if $j\notin I$, then $y_j'\in M$ moreover the inner products $y_j\cdot y_i$ for $1\leq i\leq k$, and $y_j\cdot y_i'$ for $i\in I$ are all determined by equations \eqref{4.7}-\eqref{4.8}. It follows that $y_j'$ is uniquely determined for $j\notin I$, thus the number of solutions for a fixed index set $I$ is bounded by the number of $k+l$-tuples $y_1,\ldots,y_k,y_{i_1}',\ldots y_{i_l}'$ satisfying equations 
\eqref{4.7}-\eqref{4.8}. The inner products of these vectors form a positive definite matrix, thus applying estimate \eqref{3.2'} we obtain that number of solutions is bounded by $\,C\la^{d(k+l)-(k+l)(k+l+1)}< C\,\la^{2dk-2k(2k+1)}\,$, in dimensions $d>4k$.
As the number of possible index sets $I$ depends only on $k$, the total number of linearly dependent solutions to the system \eqref{4.7}-\eqref{4.8} is also bounded by $C\,\la^{2dk-2k^2-6k+4}$.
\end{proof}

\begin{proof}[Proof of Propostion \ref{Propn1} (i)]

We use the same notation as above and assume that $\|f_1\|_\infty,\dots,\|f_{k-1}\|_\infty\leq1$. 

For any given $x\in\Z^d$ we have
\[A_\la(f_1,\ldots,f_k)(x)\leq\ \la^{-dk+k(k+1)}\ \sum_{y_k} f_k(x-y_k)S_\lm(y_k)\,\sum_{\uy_1} S_{\la^2T}(\uy_1,y_k) \]
and hence, after an application of Cauchy-Schwarz, we obtain
\[A_\la(f_1,\ldots,f_k)(x)^2 \leq\ A_*(f_k^2)(x) \ \la^{-d(2k-1)+2k(k+1)-2}  \sum_{y_k,\uy_1,\uy'_1} S_{\la^2 T}(\uy_1,y_k)\, S_{\la^2 T}(\uy'_1,y_k).\]

The sum in the expression above is the number of solutions $y_1,\ldots,y_{k-1},y_1',\ldots,y_{k-1}'\in\Z^d$ and $y_k\in\Z^d$ to the system of quadratic equations
\begin{align}\label{4.7'}
y_i\cdot y_j &=y_i'\cdot y_j'=\la^2 t_{ij},\ \text{for}\ \  1\leq i,j\leq k-1\nonumber\\
y_i\cdot y_k& = y_i'\cdot y_k = \la^2 t_{ik}, \ \text{for}\ \  1\leq i\leq k-1\\
\ y_k\cdot y_k &=\la^2 t_{kk}\nonumber.
\end{align}


If one now argues, as in the proof of part (ii) above, it follows from estimate \eqref{3.2'} that
\[\sum_{y_k,\uy_1,\uy'_1} S_{\la^2 T}(\uy_1,y_k)\, S_{\la^2 T}(\uy'_1,y_k)\leq C_{d,k,T} \, \la^{d(2k-1)-2k(k+1)+2}.\]
We choose to omit the details of this calculation.
\end{proof}

\section{A strengthening of Theorem \ref{Thm1} in high dimensions}\label{last}

If, in the proof of Proposition \ref{Propn1}, we apply H\"older's inequality with conjugate exponents $m/(m-1)$ and $m$ instead of the Cauchy-Schwarz inequality, this results in $y_1,\ldots,y_{k-1}$ and $y_1,\ldots,y_{k}$ being increased $m$-fold as opposed to being doubled, in parts (i) and (ii) respectively. 

Working through these details, which we omit (but one may consult Section \ref{general} for the analogous, and somewhat similar, details in the continuous setting), one obtains the following

\begin{propn}\label{Propn2}
Let $k\in\N$ and $\Tr=\{v_0=0,v_1,\ldots,v_k\}\subs\Z^d$ be a non-degenerate $k$-simplex.

Let $m\geq2$ be an integer and set $q=m/(m-1)$. 

\begin{itemize}
\item[(i)] If $d\geq2m(k-1)+5$, then for any $f_1,\dots,f_k:\Z^d\to\R$, one has
\bee
A_\ast(f_1,\ldots,f_k)(x)\leq C_{d,m,\Delta}\,\|f_1\|_\infty\cdots\|f_{k-1}\|_\infty \ A_\ast (|f_k|^q)(x)^{1/q}
\eee
uniformly for $x\in\Z^d$.
\vspace{3pt}
\item[(ii)]
If $d\geq2mk+3$, then for any $f_1,\dots,f_k:\Z^d\to\R$, one has
\bee
A_\ast(f_1,\ldots,f_k)(x)\leq C_{d,m,\Delta}\,A_\ast(|f_1|^q,\ldots,|f_{k-1}|^q)(x)^{1/q}\,A_\ast (|f_k|^q)(x)^{1/q}
\eee
and hence
\bee
A_\ast(f_1,\ldots,f_k)(x)\leq C_{d,m,\Delta}\,A_\ast(|f_1|^{q^{k-1}})(x)^{1/q^{k-1}}A_\ast(|f_2|^{q^{k-1}})(x)^{1/q^{k-1}}\prod_{j=3}^k A_\ast(|f_j|^{q^{k+1-j}})(x)^{1/q^{k+1-j}}
\eee
uniformly for $x\in\Z^d$.
\end{itemize}
\end{propn}

Proposition \ref{Propn2} allows us to establish a strengthening of Theorem \ref{Thm1} in high dimensions, namely Theorem \ref{Thm2} below. Before stating this result we define for each integer $k\geq2$ and $1<q\leq2$, a symmetric convex region $\mathcal{C}_{k,q}\subseteq[0,1]^k$. 
We define $\mathcal{C}_{k,q}$ to be all those points $(x_1,\dots,x_k)\in[0,1]^k$ with \[x_1+\cdots+x_k<q^{-1}+q^{-2}+\cdots+q^{-(k-1)}+q^{-(k-1)}\] that also have the property that for any $1\leq j\leq k-1$ one has $y_1+\cdots+y_j<q^{-1}+\cdots+q^{-j}$ for any choice  $\{y_1,\dots,y_j\}\subset\{x_1,\dots,x_k\}$. Note that $\mathcal{C}_{k}=\mathcal{C}_{k,2}$, and that $\mathcal{C}_{k,q}$ contains the cube $q^{-(k-1)}\cdot[0,1)^k$ which approaches $[0,1)^k$ as $q\to1$.
\begin{figure}[h]
\tdplotsetmaincoords{80}{110}
\begin{tikzpicture}[tdplot_main_coords]
\draw[->]  (0,0,0) -- (7.5,0,0) node[anchor=north]{$x_1$};
\draw
[->] (0,0,0) -- (0,4.5,0) node[anchor=west]{$x_2$};
\draw[->] (0,0,0) -- (0,0,4.5) node[anchor=east]{$x_3$};

\draw[fill=gray!15, fill opacity=0.8, style = dotted] (5,0,0) -- (5,3,0)  -- (5,3,2.8) -- (5,0,2.75) -- cycle;
\draw[fill=gray!20, fill opacity=0.8, style = dotted]  (5,3,0) -- (5,3,2.8) -- (4.4,3.55,2.80) --  (4.4,3.55,0) -- cycle;
\draw[fill=gray!25, fill opacity=0.8, style = dotted]  (4.4,3.55,2.80) --  (4.4,3.55,0) -- (0,3.4,0) -- (0,3.4,2.9)  -- cycle;
\draw[fill=gray!10, fill opacity=0.8, style = dotted] (5,3,2.8)  -- (4.8,3.2,3.55) -- (4.4,0,3.4) -- (5,0,2.75) -- cycle;
\draw[fill=gray!15, fill opacity=0.8, style = dotted] (4.4,3.55,2.80) -- (5,3,2.8) -- (4.8,3.2,3.55) -- cycle;
\draw[fill=gray!15, fill opacity=0.8, style = dotted] (4.8,3.2,3.55) -- (4.4,3.55,2.80) -- (0,3.4,2.9) -- (0,2.9,3.5) -- cycle;

\draw[fill=gray!10, fill opacity=0.8, style = dotted] (0,0,3.4) -- (0,2.9,3.5) -- (4.8,3.2,3.55) --(4.4,0,3.4) -- cycle;

\node [draw, circle, minimum width=4pt, inner sep=0pt] at (6,0,0) {};
\node [draw, color=red,circle, minimum width=4pt, inner sep=0pt] at (5,0,0) {};
\node[color=red,above left=2pt of {(6,0,0)}, outer sep=4pt] {\small$(q^{-1},0,0)$};

\node[below =2pt of {(6,1,0)}, outer sep=2pt] {\small$(1,0,0)$};
\node [draw, color=red, circle, minimum width=4pt, inner sep=0pt] at (5,3,0) {};
\node [draw, color=red, color=red, circle, minimum width=4pt, inner sep=0pt] at (5,3,2.8) {};
\node[color=red, left=2pt of {(5,3,2.8)}, outer sep=2pt] {\small$\vec{x}_1$};

\node [draw, color=red, circle, minimum width=4pt, inner sep=0pt] at (4.4,3.55,0) {};

\node [draw, color=red, circle, minimum width=4pt, inner sep=0pt] at (5,0,2.75) {};

\node [draw, color=red, circle, minimum width=4pt, inner sep=0pt] at (4.4,3.55,2.80) {};
\node[color=red, right=2pt of {(4.4,3.55,2.80)}, outer sep=2pt] {\small$\vec{x}_2$};

\node [draw, color=red, circle, minimum width=4pt, inner sep=0pt] at (4.4,0,3.4) {};

\node [draw,  color=red, circle, minimum width=4pt, inner sep=0pt] at (0,0,3.4) {};
\node[color=red,left =2pt of {(1,0,4)}, outer sep=4pt] {\small$(0,0,q^{-1})$};

\node [draw, circle, minimum width=4pt, inner sep=0pt] at (0,0,3.75) {};
\node[above right=2pt of {(0,0,3.75)}, outer sep=2pt] {\small$(0,0,1)$};

\node [draw, color=red, circle, minimum width=4pt, inner sep=0pt] at (4.8,3.2,3.55) {};
\node[color=red, above=2pt of {(4.8,3.2,3.55)}, outer sep=2pt] {\small$\vec{x}_3$};

\node [draw, color=red, circle, minimum width=4pt, inner sep=0pt] at (0,3.4,2.9) {};
\node [draw, color=red, circle, minimum width=4pt, inner sep=0pt] at (0,2.9,3.5) {};

\node [draw,  circle, minimum width=4pt, inner sep=0pt] at (0,3.7,0) {};

\node [draw, color=red, circle, minimum width=4pt, inner sep=0pt] at (0,3.4,0) {};
\node[color=red,below =2pt of {(0,4.0,-0.3)}, outer sep=4pt] {\small$(0,q^{-1},0)$};
\node[above =2pt of {(0,4.7,0.2)}, outer sep=4pt] {\small$(0,1,0)$};

\draw[color=black, style=dashed] (6,4,0) -- (6,4,3.9);

\draw[color=black, style=dashed] (6,0,0) -- (6,4,0) -- (0,3.7,0);
\draw[color=black, style=dashed] (6,0,0) -- (6,0,3.8) -- (0,0,3.75);
\draw[color=black, style=dashed] (0,0,3.75) -- (0,3.7,3.9) -- (0,3.7,0);
\draw[color=black, style=dashed] (6,0,0) -- (6,4,0) -- (0,3.7,0);
\draw[color=black, style=dashed] (6,0,3.8) -- (6,4,3.9) -- (0,3.7,3.9);

\end{tikzpicture}
\caption{Illustration of $\mathcal{C}_{3,q}$ where $\vec{x}_1=(q^{-1}, q^{-2}, q^{-2})$, $\vec{x}_2=( q^{-2}, q^{-1},q^{-2})$, and $\vec{x}_3=( q^{-2}, q^{-2},q^{-1}).$}\label{fig6}
\end{figure}
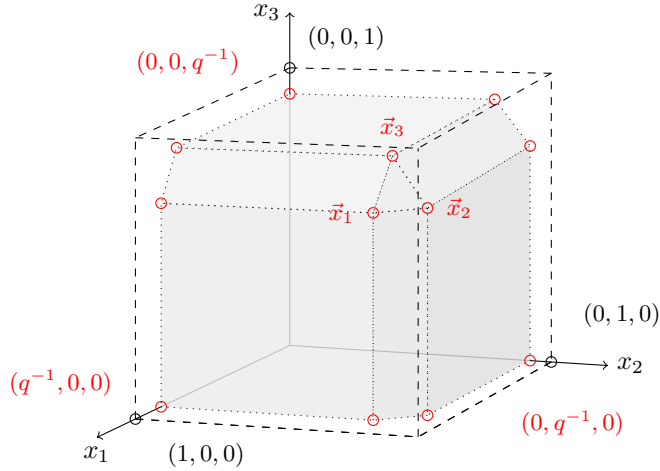


\begin{thm}\label{Thm2} Let $k\in\N$ and  $\Tr=\{v_0=0,v_1,\ldots,v_k\}\subs\Z^d$ be a non-degenerate $k$-simplex.

Let $m\geq2$ be an integer and set $q=m/(m-1)$.


\begin{itemize}
\item[(i)] If $d\geq2m(k-1)+5$, $r>q\,d/(d-2)$, and $1\leq p_1,\dots,p_k\leq\infty$ with $1/r\leq1/p_1+\cdots+1/p_k$, one has 
\bee
\|A_\ast(f_1,\ldots,f_k)\|_r \leq C_{d,m,\Tr}\, \|f_1\|_{p_1}\cdots \|f_k\|_{p_k}.
\eee


\item[(ii)] If $d\geq 2mk+3$, then for any \[r>(q^{-1}+q^{-2}+\cdots+q^{-(k-1)}+q^{-(k-1)})^{-1}\,d/(d-2) \   \text{and} \ p_1,\dots,p_k>q\,d/(d-2)\] whose reciprocals $(1/p_1,\dots,1/p_k)\in (d-2)/d\cdot\mathcal{C}_{k,q}$ and satisfy $1/r=1/p_1+\cdots+1/p_k$, 
one has 
\bee
\|A_\ast(f_1,\ldots,f_k)\|_r \leq C_{d,m,\Tr}\, \|f_1\|_{p_1}\cdots \|f_k\|_{p_k}.
\eee
\end{itemize}
\end{thm}

 For a visualization of those $p_1,\dots,p_k$ for which Theorem \ref{Thm2} gives boundedness for these discrete multilinear maximal operators, see Figures \ref{fig3} and \ref{fig4} below.

Note that Theorem \ref{Thm2} provides us with a strengthening of Theorem \ref{Thm1} (i) and (ii) for all $d\geq 6k-1$ and $d\geq 6k+3$, respectively. Note that Theorem \ref{Thm2} is of particular interest as $m\to\infty$ (and hence $d\to\infty$) for fixed $k$, since this corresponds to $q\to1$ through values of the form $m/(m-1)$ with $m\in \N$.

\begin{center}
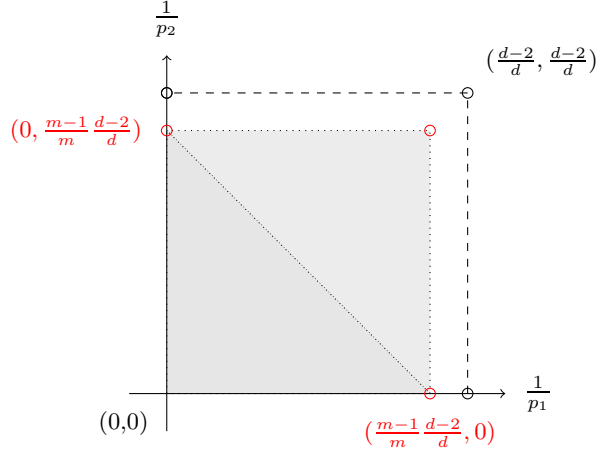

	\begin{tikzpicture}[scale=0.5]
	\draw[fill=gray!15, style = dotted] (7,0) -- (7,7) -- (0,7) -- cycle;
	\draw[fill=gray!20, style = dotted] (7,0) -- (0,7) -- (0,0) -- cycle;
	\draw[color=black,->] (-1,0) -- (9,0);
	\draw[color=black,->] (0,-1) -- (0,9);
	\draw[color=black, style=dashed] (8,0) -- (8,8);
	\draw[color=black, style=dashed] (8,8) -- (0,8);

	\node [draw, circle, minimum width=4pt, inner sep=0pt] at (8,8) {};
	\node[above right=2pt of {(8,8)}, outer sep=2pt] {\small$(\frac{d-2}{d},\frac{d-2}{d})$};
	
	
	\node [draw, color=red, circle, minimum width=4pt, inner sep=0pt] at (7,7) {};
	
	\node [draw, circle, minimum width=4pt, inner sep=0pt] at (8,0) {};

	\node [draw, color=red, circle, minimum width=4pt, inner sep=0pt] at (7,0) {};
	\node[color=red, below =2pt of {(7,0)}, outer sep=4pt] {\small$(\frac{m-1}{m}\frac{d-2}{d},0)$};
	
	\node [draw, circle, minimum width=4pt, inner sep=0pt] at (0,8) {};
	
	\node [draw, color=red, circle, minimum width=4pt, inner sep=0pt] at (0,7) {};
	\node[color=red,left=2pt of {(0,7)}, outer sep=4pt] {\small$(0,\frac{m-1}{m}\frac{d-2}{d})$};
	
	\node [draw, circle, minimum width=4pt, inner sep=0pt] at (0,8) {};
	\node[below left=2pt of {(0,0)}, outer sep=2pt] {\small(0,0)};
	
	\node[right=2pt of {(9,0)}, outer sep=2pt] {$\frac{1}{p_1}$};
	
	\node[above=2pt of {(0,9)}, outer sep=2pt] {$\frac{1}{p_2}$};
	\end{tikzpicture}
	\captionof{figure}{\color{black} Region of boundedness for the Discrete Bilinear ($k=2$) Maximal Operator in $\mathbb{Z}^d$. We obtain the lower gray triangle if $d\geq 2m+5$ and the full square provided $d\geq 4m+3$. }\label{fig3}
\end{center}


\begin{figure}[h]
\tdplotsetmaincoords{80}{110}
\begin{tikzpicture}[tdplot_main_coords]
\draw[->]  (0,0,0) -- (7.5,0,0) node[anchor=north]{$\frac{1}{p_1}$};
\draw
[->] (0,0,0) -- (0,4.5,0) node[anchor=west]{$\frac{1}{p_2}$};
\draw[->] (0,0,0) -- (0,0,4.5) node[anchor=east]{$\frac{1}{p_3}$};

\draw[fill=gray!50, fill opacity=0.8, style = dotted] (5,0,0) -- (0,0,3.4)  -- (0,3.4,0) -- cycle;

\draw[fill=gray!15, fill opacity=0.8, style = dotted] (5,0,0) -- (5,3,0)  -- (5,3,2.8) -- (5,0,2.75) -- cycle;
\draw[fill=gray!20, fill opacity=0.8, style = dotted]  (5,3,0) -- (5,3,2.8) -- (4.4,3.55,2.80) --  (4.4,3.55,0) -- cycle;
\draw[fill=gray!25, fill opacity=0.8, style = dotted]  (4.4,3.55,2.80) --  (4.4,3.55,0) -- (0,3.4,0) -- (0,3.4,2.9)  -- cycle;
\draw[fill=gray!10, fill opacity=0.8, style = dotted] (5,3,2.8)  -- (4.8,3.2,3.55) -- (4.4,0,3.4) -- (5,0,2.75) -- cycle;
\draw[fill=red!20, fill opacity=0.8, style = dotted] (4.4,3.55,2.80) -- (5,3,2.8) -- (4.8,3.2,3.55) -- cycle;
\draw[fill=gray!15, fill opacity=0.8, style = dotted] (4.8,3.2,3.55) -- (4.4,3.55,2.80) -- (0,3.4,2.9) -- (0,2.9,3.5) -- cycle;

\draw[fill=gray!10, fill opacity=0.8, style = dotted] (0,0,3.4) -- (0,2.9,3.5) -- (4.8,3.2,3.55) --(4.4,0,3.4) -- cycle;

\node [draw, circle, minimum width=4pt, inner sep=0pt] at (6,0,0) {};
\node [draw, color=red,circle, minimum width=4pt, inner sep=0pt] at (5,0,0) {};
\node[color=red,above left=2pt of {(6,0,0)}, outer sep=4pt] {\small$(\frac{m-1}{m}\frac{d-2}{d},0,0)$};

\node[below =2pt of {(6,1,0)}, outer sep=2pt] {\small$(\frac{d-2}{d},0,0)$};
\node [draw, circle, minimum width=4pt, inner sep=0pt] at (5,3,0) {};
\node [draw, color=red, circle, minimum width=4pt, inner sep=0pt] at (5,3,2.8) {};
\node[color=red, left=2pt of {(5,3,2.8)}, outer sep=2pt] {\small$\vec{x}_1$};

\node [draw, circle, minimum width=4pt, inner sep=0pt] at (4.4,3.55,0) {};

\node [draw, circle, minimum width=4pt, inner sep=0pt] at (5,0,2.75) {};

\node [draw, color=red, circle, minimum width=4pt, inner sep=0pt] at (4.4,3.55,2.80) {};
\node[color=red, right=2pt of {(4.4,3.55,2.80)}, outer sep=2pt] {\small$\vec{x}_2$};

\node [draw, circle, minimum width=4pt, inner sep=0pt] at (4.4,0,3.4) {};

\node [draw,  color=red, circle, minimum width=4pt, inner sep=0pt] at (0,0,3.4) {};
\node[color=red,left =2pt of {(1,0,4)}, outer sep=4pt] {\small$(0,0,\frac{m-1}{m}\frac{d-2}{d})$};

\node [draw, circle, minimum width=4pt, inner sep=0pt] at (0,0,3.75) {};
\node[above right=2pt of {(0,0,3.75)}, outer sep=2pt] {\small$(0,0,\frac{d-2}{d})$};

\node [draw, color=red, circle, minimum width=4pt, inner sep=0pt] at (4.8,3.2,3.55) {};
\node[color=red, above=2pt of {(4.8,3.2,3.55)}, outer sep=2pt] {\small$\vec{x}_3$};

\node [draw, circle, minimum width=4pt, inner sep=0pt] at (0,3.4,2.9) {};
\node [draw, circle, minimum width=4pt, inner sep=0pt] at (0,2.9,3.5) {};

\node [draw, circle, minimum width=4pt, inner sep=0pt] at (0,3.7,0) {};

\node [draw, color=red, circle, minimum width=4pt, inner sep=0pt] at (0,3.4,0) {};
\node[color=red,below =2pt of {(0,4.0,-0.3)}, outer sep=4pt] {\small$(0,\frac{m-1}{m}\frac{d-2}{d},0)$};
\node[above =2pt of {(0,4.7,0.2)}, outer sep=4pt] {\small$(0,\frac{d-2}{d},0)$};

\draw[color=black, style=dashed] (6,4,0) -- (6,4,3.9);

\draw[color=black, style=dashed] (6,0,0) -- (6,4,0) -- (0,3.7,0);
\draw[color=black, style=dashed] (6,0,0) -- (6,0,3.8) -- (0,0,3.75);
\draw[color=black, style=dashed] (0,0,3.75) -- (0,3.7,3.9) -- (0,3.7,0);
\draw[color=black, style=dashed] (6,0,0) -- (6,4,0) -- (0,3.7,0);
\draw[color=black, style=dashed] (6,0,3.8) -- (6,4,3.9) -- (0,3.7,3.9);

\end{tikzpicture}
\caption*{\textcolor{red}{$\vec{x}_1=(\frac{m-1}{m}\frac{d-2}{d}, (\frac{m-1}{m})^2\frac{d-2}{d}, (\frac{m-1}{m})^2\frac{d-2}{d})$}}
\caption*{\textcolor{red}{$\vec{x}_2=( (\frac{m-1}{m})^2\frac{d-2}{d}, \frac{m-1}{m}\frac{d-2}{d},(\frac{m-1}{m})^2\frac{d-2}{d})$}}
\caption*{\textcolor{red}{$\vec{x}_3=( (\frac{m-1}{m})^2\frac{d-2}{d}, (\frac{m-1}{m})^2\frac{d-2}{d},\frac{m-1}{m}\frac{d-2}{d})$
}}
\caption{Region of boundedness for the Discrete Trilinear ($k=3$) Maximal Operator in $\mathbb{Z}^d$. We obtain the dark grey tetrahedron if  $d\geq4m+5$ and the larger convex region provided $d\geq6m+3$.}\label{fig4}
\end{figure}

\begin{proof}[Proof of Theorem \ref{Thm2}]\label{cont}


We first establish part (i). Proposition \ref{Propn2} (i)  implies
\[\|A_\ast(f_1,\ldots,f_k)\|_{p_k}\leq C_{d,m,\Delta}\,\|f_1\|_\infty\cdots\|f_{k-1}\|_\infty  \|A_\ast (|f_k|^q)\|^{1/q}_{p_k/q}\leq C_{d,m,\Delta}\, \|f_1\|_\infty\cdots\|f_{k-1}\|_\infty  \|f_k\|_{p_k}\]
provided $p_k>q\,d/(d-2)$. By symmetry and interpolation we then obtain part (i) of Theorem \ref{Thm2}.

\smallskip

To establish part (ii), we note that Proposition \ref{Propn2} (ii)  ensures that
\bee
A_\ast(f_1,\ldots,f_k)(x)\leq C_{d,m,\Delta}\,A_\ast(|f_1|^{q^{k-1}})(x)^{1/q^{k-1}}A_\ast(|f_2|^{q^{k-1}})(x)^{1/q^{k-1}}\prod_{j=3}^k A_\ast(|f_j|^{q^{k+1-j}})(x)^{1/q^{k+1-j}}.
\eee
An application of H\"older, as in the proof of Theorem \ref{Thm1}, then gives
\[\|A_\ast(f_1,\ldots,f_k)\|_r\leq C_{d,m,\Delta}\|A_\ast(|f_1|^{q^{k-1}})\|_{p_1/q^{k-1}}^{1/q^{k-1}}\|A_\ast(|f_2|^{q^{k-1}})\|_{p_2/q^{k-1}}^{1/q^{k-1}}\prod_{j=3}^k \|A_\ast(|f_j|^{q^{k+1-j}})\|_{p_j/q^{k+1-j}}^{1/q^{k+1-j}}\]
whenever $1/r=1/p_1+\cdots+1/p_k$. Now if
\bee
p_1,p_2>q^{k-1}\frac{d}{d-2} \ \ \ \text{and} \ \ \ p_j>q^{k+1-j}\frac{d}{d-2} \ \ \text{for} \  \  3\leq j\leq k\eee
then by \eqref{2.3} we obtain 
\[\|A_\ast(f_1,\ldots,f_k)\|_r\leq C_{d,m,\Delta}\|f_1\|_{p_1}\cdots\|f_k\|_{p_k}\]
with $1/r=1/p_1+\cdots+1/p_k<(1/q+1/q^2+\cdots+1/q^{k-1}+1/q^{k-1})(d-2)/d$. 

Part (ii) of Theorem \ref{Thm2} now follows by symmetry and interpolation.
\end{proof}

\smallskip

\section{An example for Multilinear operators in $\Z^d$}\label{example}

Let $d\geq 2k+3$. Simple examples show that estimates $\ \|A_\ast(f_1,f_2,\ldots,f_k)\|_r \leq C\,\|f_1\|_{p_1} \|f_{p_2}\| \cdots \|f_k\|_{p_k}\,$ are not possible when $\frac{1}{r}=\frac{1}{p_1}+\ldots +\frac{1}{p_k}$, for any $r>1$, $p_1\leq \frac{d}{d-2}$, $1\leq p_i\leq\infty$. By symmetry, this shows that the reciprocals $(1/p_1,\ldots,1/p_k)$ for which the discrete maximal operator is bounded is contained in the cube $Q_{d}=(d-2)/d \cdot [0,1)^k$ and contains the smaller cube $q^{-(k-1)}\cdot Q_d$ with $q=m/m-1$, in dimensions $d>2km+2$ by Theorem \ref{Thm2}. 

Assume first that $p_i<\infty$ for $2\leq i\leq k$.
Let $f_1:=\de_0$ the point mass at the origin, and for $2\leq i\leq k$ let  $f_i(x)=|x|^{-d/p_i}\,(\log\,(|x|))^{-\frac{1}{p_i}-\tau}$ ($\tau>0$) for $|x|\geq 2$, and set $f_i(x)=0$ of $|x|\leq 1$. 
Given $x\in\Z^d$, $x\neq 0$ and $\la\in\sqrt{\N}$, we have
\[A_\la(f_1,\ldots,f_k) (x) =\,\la^{-dk+k(k+1)}\,\sum_{y_1,\ldots,y_k}
S_{\la^2 T}(y_1,y_2,\ldots,y_k)\ \de_0(x-y_1)\prod_{i=2}^k f_i(x-y_i).\]

Choosing $\la=|x|$, one has that $y_1=x$ and $|x-y_i|\geq c\,|x|$ for $2\leq i\leq k$,
for any non-zero term. Thus one estimates
\bee
A_\ast(f_1,\ldots,f_k) (x) \geq C\,|x|^{-dk+k(k+1)}\,|x|^{-d(\frac{1}{r}-\frac{1}{p_1})}\,
(\log\,(|x|))^{-(\frac{1}{r}-\frac{1}{p_1}-k\tau)}\ 
W_{|x|}(x)
\eee
where $W_{|x|}(x):=\sum_{y_2,\ldots,y_k}S_{|x|^2 T}(x,y_2,\ldots,y_k).$
We will show that 
\eq\label{5.2'}
\sum_{x\in\Z^d} W_{|x|}(x)^r= \sum_{j=0}^\infty\ \sum_{2^j\leq |x|<2^{j+1}}  W_{|x|}(x)^r
\geq \sum_{j=0}^\infty 2^{jd}\,\big(\,2^{-jd}\sum_{2^j\leq |x|<2^{j+1}}  W_{|x|}(x)\big)^r =\infty
\ee
for which it is enough to show that 
\[
2^{-jd} \sum_{2^j\leq |x|<2^{j+1}}  W_{|x|}(x) \geq C\,2^{-\frac{jd}{r}} j^{-\frac{1}{r}}.
\]

Using the notation $|x|\approx 2^j$ for $2^j\leq |x|<2^{j+1}$, one has
\eq\label{5.3'}
2^{-jd}\sum_{|x|\approx 2^j}  W_{|x|}(x) \geq c\,
2^{-jd (1+\frac{1}{r}-\frac{1}{p_1})} j^{-(\frac{1}{r}-\frac{1}{p_1}-k\tau)} 2^{-j(dk-k(k+1))} \sum_{|x|\approx 2^j}\sum_{y_2,\ldots,y_k} S_{|x|^2 T}(x,y_2,\ldots,y_k).   
\ee

Writing $\la=|x|\in\sqrt{N}$ the right side of the above expression can further estimated from below by
\[ 2^{-j(dk-k(k+1))}\sum_{2^j\leq\la <2^{j+1}} \sum_{y_1,\ldots,y_k} S_{\la^2 T}(y_1,\ldots,y_k) \geq C \sum_{2^j\leq\la <2^{j+1}} 1\, \geq\,C\,2^{2j}\]
for some constant $C=C_{d,k,T}>0$, using estimate \eqref{2.5} and the fact that there are approximately $2^{2j}$ values of $\la\in\sqrt{\N}$ satisfying $2^j\leq \la <2^{j+1}$.
This implies that for $p_1\leq d/(d-2)$, $\tau\leq 1/k p_1$ the left side of \eqref{5.3'} is bigger than 
\[C\,2^{-j(d-2-\frac{d}{p_1}+\frac{d}{r})}\,j^{-\frac{1}{r}+\frac{1}{p_1}-k\tau} \geq C\,
2^{-\frac{jd}{r}}\,j^{-\frac{1}{r}}.\]
This shows the validity of \eqref{5.2'}. The same argument works when $p_i=\infty$ for some $2\leq i\leq k$ by choosing $f_i$ to be the constant 1 function.

\section{Main results for our Multilinear Maximal Operators on $\mathbb{R}^d$}\label{cont}

For simplicity we first state our main result, in the continuous setting, in the bilinear ($k=2$) case.
\begin{thm}\label{Thm3} 
Let $\Tr=\{v_0=0,v_1,v_2\}\subs\R^d$ be a non-degenerate triangle.

Let $m\geq2$ be an integer and set $q=m/(m-1)$. 
If $d\geq 2m$, $r>q/2\cdot d/(d-1)$, and $p_1,p_2>q\,d/(d-1)$ with $1/r=1/p_1+1/p_2$, 
then
one has the estimate
\bee
\|\mathcal{A}_\ast(f_1,f_2)\|_r \leq C_{d,\Tr}\, \|f_1\|_{p_1}\|f_2\|_{p_2}.
\eee
 \end{thm}

If for each $1<q\leq2$ we define new symmetric convex region $\widetilde{\mathcal{C}}_{2,q}\subseteq[0,1]^2$ to be the convex hull of the region $\mathcal{C}_{2,q}$, as defined in Section \ref{last}, and the triangle of points  $(x_1,x_2)\in[0,1]^2$ with $x_1+x_2<1$, then combining Theorem 0 with Theorem \ref{Thm3} gives the following 
 
 \begin{cor}\label{Cor3} 
Let $\Tr=\{v_0=0,v_1,v_2\}\subs\R^d$ be a non-degenerate triangle.

Let $m\geq2$ be an integer and set $q=m/(m-1)$. 
If $d\geq 2m$, $r>q/2\cdot d/(d-1)$, and the reciprocals of $p_1$ and $p_2$ satisfy both $(1/p_1,1/p_2)\in (d-1)/d\cdot\widetilde{\mathcal{C}}_{2,q}$ and $1/r=1/p_1+1/p_2$, 
then one has the estimate
\bee
\|\mathcal{A}_\ast(f_1,\ldots,f_k)\|_r \leq C_{d,k,\Tr}\, \|f_1\|_{p_1}\cdots \|f_k\|_{p_k}.
\eee
 \end{cor}
 
 Theorem \ref{Thm3}, and hence Corollary \ref{Cor3}, are a special case of Theorem \ref{Thm3'} below, they in fact form the base case of an inductive argument that will be used to prove Theorem \ref{Thm3'}.
 
 For a visualization of those $p_1$ and $p_2$ for which Corollary \ref{Cor3} gives boundedness for these bilinear maximal operators, see Figure \ref{fig1} below.
 
 \begin{center}
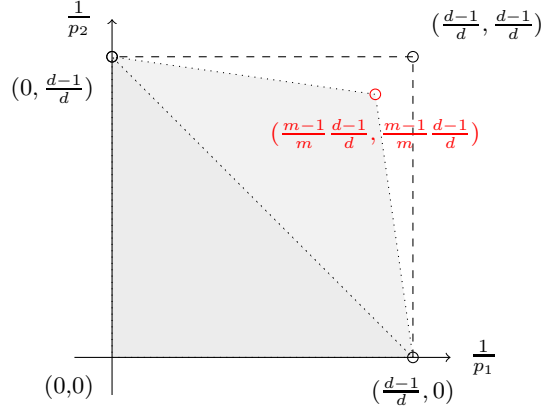

	\begin{tikzpicture}[scale=0.5]
	\draw[fill=gray!15, style = dotted] (0,0) -- (8,0)  -- (0,8) -- cycle;
	\draw[fill=gray!10, style = dotted] (8,0) -- (7,7) -- (0,8) -- cycle;
	\draw[color=black,->] (-1,0) -- (9,0);
	\draw[color=black,->] (0,-1) -- (0,9);
	\draw[color=black, style=dashed] (8,0) -- (8,8);
	\draw[color=black, style=dashed] (8,8) -- (0,8);

	\node [draw, circle, minimum width=4pt, inner sep=0pt] at (8,8) {};
	\node[above right=2pt of {(8,8)}, outer sep=2pt] {\small$(\frac{d-1}{d},\frac{d-1}{d})$};
	
	
	\node [draw, color=red, circle, minimum width=4pt, inner sep=0pt] at (7,7) {};
	\node[color=red,below=2pt of {(7,7)}, outer sep=4pt] {\small$(\frac{m-1}{m}\frac{d-1}{d},\frac{m-1}{m}\frac{d-1}{d})$};
	
	\node [draw, circle, minimum width=4pt, inner sep=0pt] at (8,0) {};
	\node[below=2pt of {(8,0)}, outer sep=2pt] {\small$(\frac{d-1}{d},0)$};
	
	\node [draw, circle, minimum width=4pt, inner sep=0pt] at (0,8) {};
	\node[below left=2pt of {(0,8)}, outer sep=2pt] {\small$(0,\frac{d-1}{d})$};
	
	\node [draw, circle, minimum width=4pt, inner sep=0pt] at (0,8) {};
	\node[below left=2pt of {(0,0)}, outer sep=2pt] {\small(0,0)};
	
	\node[right=2pt of {(9,0)}, outer sep=2pt] {$\frac{1}{p_1}$};
	
	\node[left=2pt of {(0,9)}, outer sep=2pt] {$\frac{1}{p_2}$};
	\end{tikzpicture}
	\captionof{figure}{\color{black} Region of boundedness we obtain for Bilinear Maximal Operator in $\mathbb{R}^d$, provided $d\geq2m$.}\label{fig1}
\end{center}

\bigskip

The simple observation that
\be\label{newtrivial}
\mathcal{A}_{*}(f_1,\dots,f_{k})(x)\leq \|f_k\|_\infty \,\mathcal{A}_{*}(f_1,\dots,f_{k-1})(x)
\ee
implies that bounds for multilinear maximal operators associated to $(k-1)$-simplices will always give rise to some corresponding estimates for multilinear maximal operators associated to $k$-simplices. In particular, it is easy to see using symmetry and interpolation that Corollary \ref{Cor3} gives the following non-trivial new bounds for trilinear operators.

\begin{cor}\label{Cor3'}
Let $\Tr=\{v_0=0,v_1,v_2,v_3\}\subs\R^d$ be a non-degenerate 3-simplex.

Let $m\geq2$ be an integer and set $q=m/(m-1)$. 
If $d\geq 2m$, $r>q/2\cdot d/(d-1)$, and the reciprocals of $p_1,p_2,p_3$ satisfy both $(1/p_1,1/p_2),(1/p_2,1/p_3), (1/p_1,1/p_3)\in (d-1)/d\cdot\widetilde{\mathcal{C}}_{2,q}$ and $1/r=1/p_1+1/p_2$, 
then 
\bee
\|\mathcal{A}_\ast(f_1,f_2,f_3)\|_r \leq C_{d,\Tr}\, \|f_1\|_{p_1}\|f_2\|_{p_2}\|f_3\|_{p_3}.
\eee
\end{cor}

For a visualization of those $p_1,p_2,p_3$ for which Corollary \ref{Cor3'} gives boundedness for these bilinear maximal operators, see Figure \ref{fig2'} below.

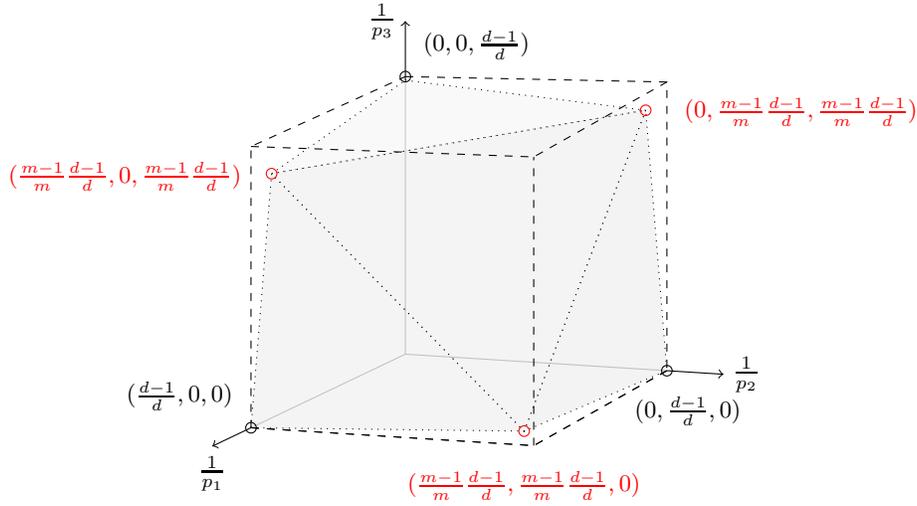
\begin{figure}[h]
\tdplotsetmaincoords{80}{110}
\begin{tikzpicture}[tdplot_main_coords]
\draw[->]  (0,0,0) -- (7.5,0,0) node[anchor=north]{$\frac{1}{p_1}$};
\draw
[->] (0,0,0) -- (0,4.5,0) node[anchor=west]{$\frac{1}{p_2}$};
\draw[->] (0,0,0) -- (0,0,4.5) node[anchor=east]{$\frac{1}{p_3}$};

\draw[fill=gray!6, fill opacity=0.8, style = dotted] (0,0,3.7) -- (0,3.4,3.5) -- (5.2,0,3.3) -- cycle;
\draw[fill=gray!10, fill opacity=0.8, style = dotted] (6,0,0) --  (5.2,0,3.3) -- (5,3.5,0) -- cycle;
\draw[fill=gray!12, fill opacity=0.8, style = dotted] (0,3.4,3.5) --  (5,3.5,0) -- (0,3.7,0) -- cycle;
\draw[fill=gray!10, fill opacity=0.8, style = dotted] (0,3.4,3.5) --  (5,3.5,0) --  (5.2,0,3.3) -- cycle;

\node [draw, circle, minimum width=4pt, inner sep=0pt] at (6,0,0) {};


\node[above left=2pt of {(6,0,0)}, outer sep=2pt] {\small$(\frac{d-1}{d},0,0)$};
\node [draw, color=red, circle, minimum width=4pt, inner sep=0pt] at (5,3.5,0) {};
\node[color=red, below=4pt of {(5,3.5,-0.2)}, outer sep=2pt] {\small$(\frac{m-1}{m}\frac{d-1}{d},\frac{m-1}{m}\frac{d-1}{d},0)$};

\node [draw, color=red, circle, minimum width=4pt, inner sep=0pt] at (5.2,0,3.3) {};
\node[color=red, left=4pt of {(5.4,0,3.3)}, outer sep=2pt] {\small$(\frac{m-1}{m}\frac{d-1}{d},0,\frac{m-1}{m}\frac{d-1}{d})$};

\node [draw, circle, minimum width=4pt, inner sep=0pt] at (0,0,3.75) {};
\node[above right=2pt of {(0,0,3.75)}, outer sep=2pt] {\small$(0,0,\frac{d-1}{d})$};

\node [draw, color=red, circle, minimum width=4pt, inner sep=0pt] at (0,3.4,3.5) {};
\node[color=red, right=4pt of {(0,3.6,3.5)}, outer sep=2pt] {\small$(0,\frac{m-1}{m}\frac{d-1}{d},\frac{m-1}{m}\frac{d-1}{d})$};

\node [draw, circle, minimum width=4pt, inner sep=0pt] at (0,3.7,0) {};

\node[below =2pt of {(0,4,0)}, outer sep=4pt] {\small$(0,\frac{d-1}{d},0)$};

\draw[color=black, style=dashed] (6,4,0) -- (6,4,3.9);

\draw[color=black, style=dashed] (6,0,0) -- (6,4,0) -- (0,3.7,0);
\draw[color=black, style=dashed] (6,0,0) -- (6,0,3.8) -- (0,0,3.75);
\draw[color=black, style=dashed] (0,0,3.75) -- (0,3.7,3.9) -- (0,3.7,0);
\draw[color=black, style=dashed] (6,0,0) -- (6,4,0) -- (0,3.7,0);
\draw[color=black, style=dashed] (6,0,3.8) -- (6,4,3.9) -- (0,3.7,3.9);

\end{tikzpicture}
\caption{Region of boundedness obtained in Corollary \ref{Cor3'} for Trilinear Maximal Operator in $\mathbb{R}^d$, provided $d\geq2m$.}\label{fig2'}
\end{figure}

\bigskip

Before we can state our generalization of Corollary \ref{Cor3} to $k$-simplices, which will in particular also strengthen Corollary \ref{Cor3'} for $3$-simplices in higher dimensions, we define inductively, for each integer $k\geq2$ and $1<q\leq2$ a new symmetric convex region $\widetilde{\mathcal{C}}_{k,q}\subseteq[0,1]^k$.

We start with $\widetilde{\mathcal{C}}_{2,q}\subseteq[0,1]^2$ as defined immediately after Theorem \ref{Thm3} above. Then, for each integer $k\geq3$ we define $\widetilde{\mathcal{C}}_{k,q}$ to be the convex hull of the points $(x_1,\dots,x_k)\in[0,1]^k$ such that $\pi_j(x_1,\dots,x_k)\in \widetilde{\mathcal{C}}_{k-1,q}$ for each $1\leq j\leq k$, where $\pi_j$ denotes the projection onto the coordinate hyperplane $x_j=0$, with the region $\mathcal{C}_{k,q}$ as defined in Section \ref{last}.

\begin{thm}\label{Thm3'} 
Let $k\in\N$ and $\Tr=\{v_0=0,v_1,\ldots,v_k\}\subs\R^d$ be a non-degenerate $k$-simplex.

Let $m\geq2$ be an integer and set $q=m/(m-1)$. 
If $d\geq mk$, then for any \[r>(q^{-1}+q^{-2}+\cdots+q^{-(k-1)}+q^{-(k-1)})^{-1}\,d/(d-1)\] and $p_1,\dots,p_k$ whose reciprocals satisfy both $(1/p_1,\dots,1/p_k)\in (d-1)/d\cdot\widetilde{\mathcal{C}}_{k,q}$ and $1/r=1/p_1+\cdots+1/p_k$, 
one has the estimate
\bee
\|\mathcal{A}_\ast(f_1,\ldots,f_k)\|_r \leq C_{d,m,\Tr}\, \|f_1\|_{p_1}\cdots \|f_k\|_{p_k}.
\eee
 \end{thm}
 
 For a visualization of those $p_1,\dots,p_k$ for which Theorem \ref{Thm3'} gives boundedness for these multilinear maximal operators, see Figure \ref{fig2} below.
 
 As with Theorem \ref{Thm2}, we note that Theorem \ref{Thm3'} is of particular interest as $m\to\infty$ (and hence $d\to\infty$) for fixed $k$, since this corresponds to $q\to1$ through values of the form $m/(m-1)$ with $m\in \N$.
 
\begin{figure}[h]
\tdplotsetmaincoords{80}{110}
\begin{tikzpicture}[tdplot_main_coords]
\draw[->]  (0,0,0) -- (7.5,0,0) node[anchor=north]{$\frac{1}{p_1}$};
\draw
[->] (0,0,0) -- (0,4.5,0) node[anchor=west]{$\frac{1}{p_2}$};
\draw[->] (0,0,0) -- (0,0,4.5) node[anchor=east]{$\frac{1}{p_3}$};


\draw[fill=gray!17, fill opacity=0.8, style = dotted]  (5,3.5,0) -- (5,3,2.8) -- (4.4,3.55,2.80) --  (5,3.5,0) -- cycle;
\draw[fill=gray!10, fill opacity=0.8, style = dotted] (5,3,2.8)  -- (4.8,3.2,3.55) -- (5.2,0,3.3) -- (5.2,0,3.3) -- cycle;
\draw[fill=red!17, fill opacity=0.8, style = dotted] (4.4,3.55,2.80) -- (5,3,2.8) -- (4.8,3.2,3.55) -- cycle;
\draw[fill=gray!15, fill opacity=0.8, style = dotted] (4.8,3.2,3.55) -- (4.4,3.55,2.80) -- (0,3.4,3.5) -- cycle;



\draw[fill=gray!15, fill opacity=0.8, style = dotted] (6,0,0) -- (5,3.5,0)  -- (5,3,2.8) -- (6,0,0) -- cycle;
\draw[fill=gray!10, fill opacity=0.8, style = dotted] (6,0,0) -- (5,3,2.8)  -- (5.2,0,3.3) -- (6,0,0) -- cycle;
\draw[fill=gray!17, fill opacity=0.8, style = dotted]  (4.4,3.55,2.80) --  (5,3.5,0) -- (0,3.7,0) -- cycle;
\draw[fill=gray!15, fill opacity=0.8, style = dotted]  (4.4,3.55,2.80) -- (0,3.4,3.5)  -- (0,3.7,0) -- cycle;
\draw[fill=gray!10, fill opacity=0.8, style = dotted] (0,0,3.75) -- (4.8,3.2,3.55) -- (0,3.4,3.5) -- cycle;
\draw[fill=gray!8, fill opacity=0.8, style = dotted] (0,0,3.75) -- (4.8,3.2,3.55) -- (5.2,0,3.3) -- cycle;

\node [draw, circle, minimum width=4pt, inner sep=0pt] at (6,0,0) {};


\node[above left=2pt of {(6,0,0)}, outer sep=2pt] {\small$(\frac{d-1}{d},0,0)$};
\node [draw, color=red, circle, minimum width=4pt, inner sep=0pt] at (5,3.5,0) {};
\node[color=red, below=4pt of {(5,3.5,-0.2)}, outer sep=2pt] {\small$(\frac{m-1}{m}\frac{d-1}{d},\frac{m-1}{m}\frac{d-1}{d},0)$};

\node [draw, color=red, circle, minimum width=4pt, inner sep=0pt] at (5,3,2.8) {};
\node[color=red, left=2pt of {(5,3,2.8)}, outer sep=2pt] {\small$\vec{x}_1$};



\node [draw, color=red, circle, minimum width=4pt, inner sep=0pt] at (4.4,3.55,2.80) {};
\node[color=red, right=2pt of {(4.4,3.55,2.80)}, outer sep=2pt] {\small$\vec{x}_2$};

\node [draw, color=red, circle, minimum width=4pt, inner sep=0pt] at (5.2,0,3.3) {};
\node[color=red, left=4pt of {(5.4,0,3.3)}, outer sep=2pt] {\small$(\frac{m-1}{m}\frac{d-1}{d},0,\frac{m-1}{m}\frac{d-1}{d})$};


\node [draw, circle, minimum width=4pt, inner sep=0pt] at (0,0,3.75) {};
\node[above right=2pt of {(0,0,3.75)}, outer sep=2pt] {\small$(0,0,\frac{d-1}{d})$};

\node [draw, color=red, circle, minimum width=4pt, inner sep=0pt] at (4.8,3.2,3.55) {};
\node[color=red, above=2pt of {(4.8,3.2,3.55)}, outer sep=2pt] {\small$\vec{x}_3$};

\node [draw, color=red, circle, minimum width=4pt, inner sep=0pt] at (0,3.4,3.5) {};
\node[color=red, right=4pt of {(0,3.6,3.5)}, outer sep=2pt] {\small$(0,\frac{m-1}{m}\frac{d-1}{d},\frac{m-1}{m}\frac{d-1}{d})$};

\node [draw, circle, minimum width=4pt, inner sep=0pt] at (0,3.7,0) {};

\node[below =2pt of {(0,4,0)}, outer sep=4pt] {\small$(0,\frac{d-1}{d},0)$};

\draw[color=black, style=dashed] (6,4,0) -- (6,4,3.9);

\draw[color=black, style=dashed] (6,0,0) -- (6,4,0) -- (0,3.7,0);
\draw[color=black, style=dashed] (6,0,0) -- (6,0,3.8) -- (0,0,3.75);
\draw[color=black, style=dashed] (0,0,3.75) -- (0,3.7,3.9) -- (0,3.7,0);
\draw[color=black, style=dashed] (6,0,0) -- (6,4,0) -- (0,3.7,0);
\draw[color=black, style=dashed] (6,0,3.8) -- (6,4,3.9) -- (0,3.7,3.9);

\end{tikzpicture}
\smallskip
\caption*{\textcolor{red}{$\vec{x}_1=(\frac{m-1}{m}\frac{d-1}{d}, (\frac{m-1}{m})^2\frac{d-1}{d}, (\frac{m-1}{m})^2\frac{d-1}{d})$}}
\caption*{\textcolor{red}{$\vec{x}_2=( (\frac{m-1}{m})^2\frac{d-1}{d}, \frac{m-1}{m}\frac{d-1}{d},(\frac{m-1}{m})^2\frac{d-1}{d})$}}
\caption*{\textcolor{red}{$\vec{x}_3=( (\frac{m-1}{m})^2\frac{d-1}{d}, (\frac{m-1}{m})^2\frac{d-1}{d},\frac{m-1}{m}\frac{d-1}{d})$
}}
\caption{Region of boundedness we obtain for Trilinear Maximal Operator in $\mathbb{R}^d$, provided $d\geq3m$.}\label{fig2}
\end{figure}

\subsection{Proof of Theorem \ref{Thm3'}}

As in the discrete case, the crucial ingredient in our proof of Theorem \ref{Thm3'} is pointwise estimates for  $\mathcal{A}_\ast(f_1,\dots f_k)$ in terms of the spherical maximal operator applied to appropriate powers of the functions $|f_j|$, specifically

\begin{propn}\label{finalprop} Let $k,m\geq 2$ be integers and $\Tr=\{v_0=0,v_1,\ldots,v_k\}\subs\R^d$ be a non-degenerate $k$-simplex with $d\geq km$.
Then for any $f_1,\ldots,f_k:\R^d\to\R$, one has
\bee
\mathcal{A}_\ast(f_1,\ldots,f_k)(x)\leq C_{d,m,\Delta}\,\mathcal{A}_\ast(|f_1|^q,\ldots,|f_{k-1}|^q)(x)^{1/q}\,\mathcal{A}_\ast (|f_k|^q)(x)^{1/q}
\eee
and hence
\bee
\mathcal{A}_\ast(f_1,\ldots,f_k)(x)\leq C_{d,m,\Delta}\,\mathcal{A}_\ast(|f_1|^{q^{k-1}})(x)^{1/q^{k-1}}\mathcal{A}_\ast(|f_2|^{q^{k-1}})(x)^{1/q^{k-1}}\prod_{j=3}^k \mathcal{A}_\ast(|f_j|^{q^{k+1-j}})(x)^{1/q^{k+1-j}}
\eee
uniformly for $x\in\R^d$, where $q=m/(m-1)$. 
\end{propn}

We will prove Proposition \ref{finalprop} in Section \ref{Proof of final prop} below. As   in the discrete case, this proposition quickly implies Theorem \ref{Thm3'}, however in this setting we proceed by induction. 

Note that Corollary \ref{Cor3} constitutes the base $k=2$ case of Theorem \ref{Thm3'}.  We are thus required to first prove Theorem \ref{Thm3}. An application of H\"older to the  $k=2$ case of Proposition \ref{finalprop} gives
\[\|\mathcal{A}_\ast(f_1,f_2)\|_r\leq C_{d,m,\Delta}\|\mathcal{A}_\ast(|f_1|^{q})\|_{p_1/q}^{1/q}\|\mathcal{A}_\ast(|f_2|^{q})\|_{p_2/q}^{1/q}\]
whenever $1/r=1/p_1+1/p_2$. Now if
$p_1,p_2>q\,d/(d-1)$, then by \eqref{2.3'} we obtain 
\[\|\mathcal{A}_\ast(f_1,f_2)\|_r\leq C_{d,m,\Delta}\|f_1\|_{p_1}\|f_2\|_{p_2}\]
with $1/r=1/p_1+1/p_2<2/q\cdot (d-1)/d$. This establishes Theorem \ref{Thm3}.

We now let $k\geq3$ and assume that Theorem \ref{Thm3'} holds for $k-1$. An application of H\"older, as in the proof of Theorem \ref{Thm2}, gives
\[\|\mathcal{A}_\ast(f_1,\ldots,f_k)\|_r\leq C_{d,m,\Delta}\|\mathcal{A}_\ast(|f_1|^{q^{k-1}})\|_{p_1/q^{k-1}}^{1/q^{k-1}}\|\mathcal{A}_\ast(|f_2|^{q^{k-1}})\|_{p_2/q^{k-1}}^{1/q^{k-1}}\prod_{j=3}^k \|\mathcal{A}_\ast(|f_j|^{q^{k+1-j}})\|_{p_j/q^{k+1-j}}^{1/q^{k+1-j}}\]
whenever $1/r=1/p_1+\cdots+1/p_k$. Now if
\bee
p_1,p_2>q^{k-1}\frac{d}{d-1} \ \ \ \text{and} \ \ \ p_j>q^{k+1-j}\frac{d}{d-1} \ \ \text{for} \  \  3\leq j\leq k\eee
then by \eqref{2.3'} we obtain 
\[\|\mathcal{A}_\ast(f_1,\ldots,f_k)\|_r\leq C_{d,m,\Delta}\|f_1\|_{p_1}\cdots\|f_k\|_{p_k}\]
with $1/r=1/p_1+\cdots+1/p_k<(1/q+1/q^2+\cdots+1/q^{k-1}+1/q^{k-1})(d-1)/d$. 

Theorem \ref{Thm3'} now follows for all $(1/p_1,\dots,1/p_k)\in (d-1)/d\cdot\mathcal{C}_{k,q}$ by symmetry and interpolation, and for all such reciprocals in $(d-1)/d\cdot\widetilde{\mathcal{C}}_{k,q}$ by further interpolation with the estimates one obtains using observation (\ref{newtrivial}) and the inductive hypothesis.\qed

\section{Proof of Proposition \ref{finalprop}}\label{Proof of final prop}

Let $\Tr=\{v_0=0,v_1,\ldots,v_k\}\subseteq\R^d$ be a non-degenerate $k$-simplex and $T=T_\Tr= (t_{ij})_{1\leq i,j\leq k}$ with entries $t_{ij}:= v_i\cdot v_j$ denote the associated \emph{inner product matrix}. Recall that $T$ is a positive semi-definite matrix and that $T$ is positive definite if and only if $\Tr$ is non-degenerate. In fact one has that 
\eq\label{2.4'}
\det T = \text{vol}(v_1,\ldots,v_k)^2
\ee
where $\text{vol}(v_1,\ldots,v_k)$ denotes the volume of the parallelepiped spanned by the vectors $v_1,\ldots,v_k$.

It is easy to see that $\Tr'\simeq \la\Tr$, with $\Tr'=\{y_0=0,y_1,\ldots,y_k\}$, if and only if 
\eq\label{2.4''}
y_i\cdot y_j=\la^2 t_{ij}\quad \text{for all}\quad 1\leq i,j\leq k.
\ee
Thus the \emph{configuration space} of the isometric copies of the simplex $\Tr$ is the algebraic set
\bee
S_\Tr := \{(y_1,\ldots,y_k)\in\R^{dk}:\ y_i\cdot y_j- t_{ij}=0,\quad \textit{for}\quad 1\leq i,j\leq k\}.
\eee

Writing $\uy=(y_1,\ldots,y_k)$ and $f_{ij}(\uy)=y_i\cdot j-t_{ij}$ the set $S_\Tr$
is the common zero set of the family of polynomials $\FF=\{f_{ij}\}$ and is a smooth $d-k(k+1)/2$-dimensional sub-manifold of $\R^{dk}$, in dimensions $d\geq k$. Indeed, for any $\uy=(y_1,\ldots,y_k)\in S_\Tr$ the vectors $y_1,\ldots,y_k$ are linearly independent, and then it is easy to see that the gradient vectors $\nabla f_{ij}(\uy)$ are also linearly independent thus the algebraic set $S_\Tr$ has no singular points.

To any algebraic set $S$, defined as the zero set of a family of polynomials $\FF$, which has a non-singular point one may associate measure $\om_\FF$ supported on $S$, also referred to as the Gelfand-Leray measure, this is a natural analogue to the counting measure used in the discrete setting, see Birch \cite{Birch}. In Section \ref{appendix} we provide a self-contained discussion of properties of this measure that are needed for our results. 

In particular, for any family of functions $f_1,\dots,f_k:\R^d\to\R$ with $d\geq k+1$ and $\la>0$, our normalized multilinear averages
\bee
\mathcal{A}_{\lm}(f_1,\dots,f_{k})(x)= \int_{SO(d)} f_1(x+\lm \cdot U(v_1))\cdots f_k(x+\lm\cdot U(v_k))\,d\mu(U)\eee
satisfy
\be\label{newdefn}
\mathcal{A}_{\lm}(f_1,\dots,f_{k})(x)= C_{d,\Delta}\,\int_{y_1,\ldots,y_k} f_1(x+\la y_1)\ldots f_k(x+\la y_k)\,d\om_\FF (y_1,\ldots,y_k)
\ee
for some absolute constant $C_{d,\Delta}>0$. This crucial observation follows by writing 
\[
\mathcal{A}_{\lm}(f_1,\dots,f_{k})(x)= \int f_1(x+\lm y_1)\cdots f_k(x+\lm y_k)\,d\si^{d-1}(y_1)\,d\si^{d-2}_{y_1}(y_2)\ldots d\si^{d-k}_{y_1,\ldots,y_{k-1}}(y_k)
,\]
with $d\si^{d-j}_{y_1,\ldots,y_{j-1}}(y)$ being the normalized surface area measure on the sphere $S_{y_1,\ldots,y_{j-1}}$ given by the equations $|y-y_i|^2=t_{ij}$ for $0\leq i<j$, as a special case of (\ref{spheres}), see Section 10. 



In order to aid the exposition we first prove Proposition \ref{finalprop} in special case when $k=m=2$ below, delaying the proof of the general case until Section \ref{general}.

\subsection{Proof of Proposition \ref{finalprop}: Special case when $k=m=2$} 

Recall that our goal is to show that
\be\label{babygoal}
|\mathcal{A}_\ast (f_1,f_2)(x)| \leq C_{d,\Tr} \,\mathcal{A}_\ast(f_1^2)^{1/2}(x) \mathcal{A}_\ast(f_2^2)^{1/2}(x),
\ee
uniformly for $x\in\R^d$.
Towards this end we first note that
\be\label{fub}\int\int f_1(x+\la y_1) f_2(x+\la y_2)\,d\om_\FF(y_1,y_2)
= \int\int f_1(x+\la y_1) f_2(x+\la y_2)\,d\om_{\FF_{2,y_1}}(y_2)\,d\om_{\FF_1}(y_1)\ee
where $\FF=\{f_{11},f_{12},f_{22}\}$, $\FF_1=\{f_{11}\}$ and $\FF_{1,y_1}=\{f_{12},f_{22}\}$ considered as function of $y_2$ for fixed $y_1$. 

Indeed, since the function $f_{11}(y_1,y_2)=|y_1|^2-t_{11}$ depends only on the variable $y_1$ and the partition of variables $z=y_1$ and $y=y_2$ is admissible,  Fubini  \eqref{5.4} applies, see Section \ref{appendix}. Furthermore, we note that $\om_{\FF_1}=2^{-1}\,\si_1$, where $\si_1$ denotes the surface area measure on the sphere $|y_1|^2=t_{11}$.

Observation (\ref{fub}) allows one to apply Cauchy-Schwarz in the $y_1$ variable, which in light of (\ref{newdefn}) yields
\be\label{3.2}
\begin{aligned}
\mathcal{A}_\la(f_1,f_2)(x)^2 &\ls_{d,\Delta} \int_{y_1} f_1^2(x+\la y_1)\,d\si_1(y_1)\times \int_{y_1}\Bigl(\,\int_{y_2} f_2(x+\la y_2)\,d\om_{\FF_{2,y_1}}(y_2)\Bigr)^2\,d\si_1(y_1)\\
& \ls_{d,\Delta} \mathcal{A}_\ast (f_1^2)(x)\,\mathcal{B}_\la(f_2,f_2)(x)
\end{aligned}
\ee
where 
\eq\label{3.3}
\mathcal{B}_\la(f_2,f_2)(x) =  \int_{y_2,y_2'} f_2(x+\la y_2) f_2(x+\la y_2')\ w (y_2,y_2')\,d\si_2(y_2)\,d\si_2(y_2')
\ee
with $\si_2$ denoting the surface area measure on the sphere $\{y\in\R^d\,:\,|y|^2=t_{22}\}$ and weight function
\eq\label{3.5}
w (y_2,y_2') = \int_{y_1}\,d\om_{\FF_{y_2,y_2'}}(y_1)
\ee
where $\om_{\FF_{y_2,y_2'}}$ is the Gelfand-Leray measure defined by the system $\FF_{y_2,y_2'}=\{f_{11}(y_1),f_{12}(y_1,y_2),f_{12}(y_1,y_2')\}$. These facts, specifically the form of \eqref{3.5}, follows since the partition of variables $z=(y_2,y_2')$ and $y=y_1$ is admissible for the system \[\FF'=\{f_{11}(y_1),f_{12}(y_1,y_2),f_{12}(y_1,y_2'),f_{22}(y_2),f_{22}(y_2')\}.\] 

Applying Cauchy-Schwarz once more, we obtain 
\be\label{3.6}
\begin{aligned}
\mathcal{B}_\la(f_2,f_2)^2(x) &\leq \Bigl(\int_{y_2} f_2^2(x+\la y_2)\,d\si_2(y_2)\Bigr)^2\,\Bigl(\int_{y_2,y_2'} w (y_2,y_2')^2\,d\si_2(y_2)\,d\si_2(y_2')\Bigr)\\
& \leq\, \mathcal{A}_\ast(f_2^2)(x)^2\ \int_{y_2,y_2'} w(y_2,y_2')^2\,d\si_2(y_2)\,d\si_2(y_2').
\end{aligned}
\ee
Thus by \eqref{3.2} and \eqref{3.6}, to prove inequality \eqref{babygoal} it is enough to show that 
\eq\label{3.7}
\int_{y_2,y_2'} w(y_2,y_2')^2\,d\si_2(y_2)\,d\si_2 (y_2') <\,\infty.
\ee

The set $S_{\FF_{y_2,y_2'}}$ is the intersection of three spheres, whose surface are measure is either zero or $c_d r^{d-3}$, with $\,r:= \text{dist}(y_1,\text{span}\{y_2,y_2'\})$ being its radius, for any $y_1\in S_{\FF_{y_2,y_2'}}$. Note that 
\[r = \text{vol}(y_1,y_1-y_2,y_1-y_2')/ \text{vol}(y_2,y_2'),\]
thus by \eqref{2.4'}, \eqref{3.5} and \eqref{spheres} we have that 
\eq\label{3.8}
w(y_2,y_2')= c_d\,\text{vol}(y_1,y_1-y_2,y_1-y_2')^{d-4}\,\text{vol}(y_2,y_2')^{-d+3}
\leq c_{d,T}\,\text{vol}(y_2,y_2')^{-1}.
\ee

To show \eqref{3.7}, note that for given $y_2\in S_2 :=\{y\in \R^d: |y|^2=t_{22}\}$ and $j\in\N$ we have that 
\[\si_2(\{y_2'\in S_2:\ 2^{-j}\leq \text{vol}(y_2,y_2') < 2^{-j+1}\}) \leq c_{d,T}\,
2^{-(d-1)j}.\]
Thus, in dimensions $d\geq 4$, we have
\[
\pushQED{\qed} 
\int_{y_2,y_2'} w(y_2,y_2')^2\,d\si(y_2)\,d\si(y_2') \leq c_{d,T}\,
\sum_{j\geq 0} 2^{2j}\, 2^{-(d-1)j} < \infty.\qedhere \popQED\]

\subsection{Proof of Proposition \ref{finalprop}: General case}\label{general}

We will now show that inequality \eqref{babygoal} can be improved in high dimensions, specifically when $d\geq 2m$, by replacing the Cauchy-Schwarz inequality by H\"{o}lder's inequality applied with conjugate exponents $q=m/(m-1)$ and $m$ for any positive integer $m>2$. This results in increasing the variable $y_2$ $m$-fold as opposed to being doubled in the proof of the special case when $k=2$ and $m=2$ presented above. We will simultaneously also generalize this result to $k$-simplices.

Let $k,m\geq 2$ be integers and $\Tr=\{v_0=0,v_1,\ldots,v_k\}\subs\R^d$ be a non-degenerate $k$-simplex. Recall that our goal is to show that for any $f_1,\ldots,f_k:\R^d\to\R$, one has
\be\label{biggoal}
\mathcal{A}_\ast(f_1,\ldots,f_k)(x)\leq C_{d,m,\Delta}\,\mathcal{A}_\ast(|f_1|^q,\ldots,|f_{k-1}|^q)(x)^{1/q}\,\mathcal{A}_\ast (|f_k|^q)(x)^{1/q}
\ee
uniformly for $x\in\R^d$, where $q=m/(m-1)$.

Let $T=T_\Tr$ denote the inner product matrix of the simplex $\Tr=\{v_0=0,v_1,\ldots.v_k\}$. Consider the partition of variables $\uy=(\uy_1,y_k)$ with $\uy_1=(y_1,\ldots,y_{k-1})$ which is admissible for the system 
\[\F=\{f_{ij}(\uy):\ 1\leq i\leq j\leq k\}\] with $f_{ij}(\uy)=y_i\cdot y_j-t_{ij}.$ Indeed, the gradients of the family of functions $\F_1=\{f_{ij}(\uy_1),\ 1\leq i\leq j\leq k-1\}$ are linearly independent at any point $\uy_1\in S_{\F_1}$, moreover the gradients $\nabla_{y_k} f_{ik,\uy_1}$ of the system 
\[\F_{\uy_1}=\{ f_{ik,\uy_1}(y_k)=y_i\cdot y_k-t_{ik},\ 1\leq i\leq k\}\] are also linearly independent at any point $\uy\in S_\F$. Thus one may proceed in light of  (\ref{newdefn}), using H\"older's inequality in the $\uy_1$ variable, to obtain
\be
\begin{aligned}
\mathcal{A}_\la(f_1,\ldots,f_k)(x) &\ls_{d,\Delta} \Bigl(\int_{\uy_1} 
\bigl|f_1(x+\la y_1)\cdots f_{k-1}(x+\la y_{k-1})\bigr|^q\,d\om_{\F_1}(\uy_1)\Bigr)^{1/q}
\\ &\quad\quad\quad\quad\quad\quad\quad\quad\quad\quad\quad\quad\times \Bigl(\int_{\uy_1}\Bigl(\int_{y_k} |f_k(x+\la y_k)|\,d\om_{\FF_{\uy_1}}(y_k)\Bigr)^m\,d\om_{\F_1}(\uy_1)\Bigr)^{1/m}\\
& \ls_{d,\Delta} \mathcal{A}_\ast(|f_1|^q,\ldots,|f_{k-1}|^q)(x)^{1/q} \,\mathcal{B}_\la(|f_k|,\dots,|f_k|)(x)^{1/m}
\end{aligned}
\ee
where 
\be\label{4.2}
\begin{aligned}
\mathcal{B}_\la(|f_k|,\ldots,|f_k|)(x) &=  
\int_{\uy_1}\int_{z_1,\ldots,z_m}  |f_k(x+\la z_1)|\cdots |f_k(x+\la z_m)| 
\,d\om_{\FF_{\uy_1}}(z_1)\cdots \,d\om_{\FF_{\uy_1}}(z_m)\,d\om_{\F_1}(\uy_1)\\
&= \int_{z_1,\ldots,z_m} |f_k(x+\la z_1)|\cdots |f_k(x+\la z_m)|\ w (z_1,\ldots, z_m)\ d\si_k (z_1)\,\cdots\,d\si_k(z_m)
\end{aligned}
\ee
with $\si_k$ denoting the surface area measure on the sphere $\{z\in\R^d:\ |z|^2=t_{kk}\}$. In this more general setting the weight function now takes the form
\be\label{**}
w (z_1,\ldots,z_m) = \int_{\uy_1} \, d\om_{\FF_{z_1,\ldots,z_m}}(\uy_1)
\ee
where $\om_{\FF_{\uz}}$, with $\uz=(z_1,\ldots,z_m)$, is the Gelfand-Leray measure defined by the system 
\eq\label{4.4}
\FF_{\uz} = \{f_{ij,\uz}(\uy_1)=y_i\cdot y_j-t_{ij},\ 1\leq i\leq j\leq k-1,\ h_{il,\uz} (\uy_1) = y_i\cdot z_l-t_{ik},\ 1\leq i\leq k-1,\,1\leq l\leq m\}.\ee

This can be justified by showing that both integrals in (\ref{4.2}) are equal to the integral 
\eq\label{*}
\int |f_k(x+\la z_1)|\cdots |f_k(x+\la z_m)|\, d\om_{\widetilde{\FF}}(\uy_1,\uz),
\ee
with $\om_{\widetilde{\FF}}$ being the Gelfand-Leray measure of the system 
\eq\label{4.5}
\widetilde{\FF} =
\{f_{ij}(\uy_1)=y_i\cdot y_j-t_{ij},\ h_{il} (\uy_1,\uz) = y_i\cdot z_l-t_{ik},\ g_l(\uz)=z_l\cdot z_l-t_{kk}\}.\ee

To verify this,  we show that both partition of the variables $(\uy_1,\uz)$ and $(\uz,\uy_1)$ are admissible for the system \eqref{4.5}. Given $\uy_1 \in S_{\F_1}$, the vectors $z_1,\ldots,z_m$ have to be on the same $(d-k)$-dimensional sphere, which is the intersection of the $(d-k+1)$-dimensional affine subspace $M=\{z:\ z\cdot y_i=t_{ik},\,1\leq i<k\}$ and the sphere $S^{d-1}=\{z:\ |z|^2=t_{kk}\}$. If $d\geq k+m$ then one may choose $z_1,\ldots,z_m$ so that the vectors $y_1,\ldots,y_{k-1},z_1,\ldots,z_m$ are linearly independent. 
Then the gradient vectors $\nabla_z g_l(\uz)=2 z_l$ , as well as for given $1\leq i<k$ the gradient vectors $\nabla_{y_i} f_{ij} (\uy') = y_j$ and $\nabla_{y_i} h_{il}(\uy')=z_l$, are clearly linearly independent, showing that the partition of  coordinates $(u,z)$ is admissible, see Section \ref{appendix}. By a similar argument, which we omit, one shows that the partition of coordinates $(\uz,\uy_1)$ is admissible at any point of $S_{\widetilde{\FF}}$ where each vector $z_i$ is linearly independent of the vectors $y_1,\ldots,y_{k-1}$. 

Note that by removing an algebraic set of measure zero the integrals in \eqref{4.2} and \eqref{*} may be defined over points $(\uy_1,\uz)\in S_{\widetilde{\FF}}$ such that the vectors $y_1,\ldots,y_{k-1},z_1,\ldots,z_m$ are linearly independent. Thus in particular the second integral in \eqref{4.2} is restricted to linearly independent $m$-tuples $\uz=(z_1,\ldots,z_m)$.

Applying H\"{o}lder's inequality once more, we obtain 
\[
\mathcal{B}_\la(|f_k|,\ldots,|f_k|)(x)  \leq \mathcal{A}_\ast(|f_k|^q)(x)^{m/q}\ \left(\int_{z_1,\ldots,z_m} w (z_1,\ldots,z_m)^m\,d\si_k(z_1)\cdots d\si_k(z_m)\right)^{1/m}
\]
and hence (\ref{biggoal}) will follow if we can show that
\eq\label{4.7'}
\int_{z_1,\ldots,z_m} w(z_1,\ldots,z_m)^m\,d\si_k (z_1)\cdots d\si_k  (z_m) <\,\infty.\ee

Estimate (\ref{4.7'})  follows immediately from the following two lemmas.

\begin{lem}\label{geo} Let $k\geq 2,\, m\geq 2,\, d\geq m+k$ and let $z_1,\ldots,z_m\in \R^d$ be linearly independent vectors such that the algebraic set $S_{\F_{\uz}}$ has a non-singular point. Then one has
\eq\label{4.8''}
w(z_1,\ldots,z_m) \leq C_{d,m,T}\,\text{\em{vol}}(z_1,\ldots,z_m)^{-k+1}.
\ee
\end{lem}
\begin{lem}\label{withs} Let $s\geq 1, m\geq 1$ and $d\geq m+s$. Then one has 
\eq\label{4.11}
\int \text{\em{vol}}(z_1,\ldots,z_m)^{-s}\,d\si_k(z_1)\ldots d\si_k(z_m) < \infty
\ee
where the integration is taken over linearly independent $m$-tuples $\uz=(z_1,\ldots,z_m)$.
\end{lem}

Indeed (\ref{4.7'}) follows immediately from \eqref{4.8''} by taking $s=m(k-1)$ in \eqref{4.11}. \qed

\begin{proof}[Proof of Lemma \ref{withs}] For $m=1$ this is immediate. For $m\geq2$, given $z_1,\ldots,z_{m-1}$, note that 
\[
\si (\{z_m\in S^{d-1}:\ 2^{-j} \leq \text{dist}(z_m, \text{span}(z_1,\ldots,z_{m-1}) < 2^{-j+1}\}) \leq C_{d,m}\, 2^{-(d-m+1)j}.
\]
Thus
\bee
\begin{aligned}
\int \text{vol}(z_1,\ldots,z_m)^{-s}\,&d\si_k(z_1)\cdots d\si_k(z_m) \\
&\leq C_{d,m}\,
\sum_{j\geq 0}  2^{-(d-m-s+1)j}\, \int \text{vol}(z_1,\ldots,z_{m-1})^{-s}\,d\si_k(z_1)\cdots d\si_k(z_{m-1}).
\end{aligned}
\eee
Estimate \eqref{4.11} follows by induction on $m$.
\end{proof}

\begin{proof}[Proof of Lemma \ref{geo}] Recall that integration in \eqref{**} is restricted to non-singular points $\uy_1=(y_1,\ldots,y_{k-1}) \in S_{\F_{\uz}}=\{\F_{\uz}=0\}$ such that the vectors $z_1,\ldots,z_m,y_1,\ldots,y_{k-1}$ are linearly independent. For given $1\leq i<k$, let 
\[r_i(\uy_1):= \text{dist}(y_i,\,\text{span}(z_1,\ldots,z_m,y_1,\ldots,y_{i-1}))\]
denote the distance from the point $y_i$ to the subspace spanned by the vectors $y_1,\ldots,y_{i-1},z_1,\ldots,z_m$. Note that $r_i(\uy_1)>0$, in fact
\eq\label{4.9''}
r_i(\uy_1) = \text{vol}(y_i,y_{i-1},\ldots,y_1,z_1,\ldots,z_m)/\text{vol}(y_{i-1},
\ldots,y_1,z_1,\ldots,z_m).
\ee

First we show by induction on $k$ that 
\eq\label{4.10}
d\om_{\F_{\uz}}(\uy_1) = c_{d,m,T}\,\prod_{i=1}^{k-1} r_i (\uy_1)^{d-m-k}\, 
\text{vol}(z_1,\ldots,z_m)^{-k+1}\, d\si^{d-m-1}(y_1)\,d\si^{d-m-2}(y_2)\cdots d\si^{d-m-k}(y_{k-1})
\ee
where $d\si^j$ denoting the normalized surface area measure on the $j$-dimensional unit sphere in $\R^d$.

Note that in order to verify  identity \eqref{4.10} when $k=2$ we must show that
\[d\om_{\F_{\uz}}(y_1) = c_{d,m,T}\ r_1(y_1)^{d-m-2}\, \text{vol}(z_1,\ldots,z_m)^{-1}\  d\si^{d-m-1}(y_1).
\]
The algebraic set $S_{\F_{\uz}}$ is the intersection of $m+1$ spheres centered at the points $0,z_1,\ldots,z_m$ and hence is a $(d-m-1)$-dimensional sphere of radius $r_1(y_1)$ given by \eqref{4.9''}. By \eqref{5.5} we have that 
\[
d\om_{\F_{\uz}}(y_1) = c_{d,m,T}\, \text{vol}(y_1,y_1-z_1,\ldots,y_1-z_m)^{-1} r_1(y)^{d-m-1}\ d\si^{d-m-1}(y_1)
\]
and \eqref{4.10} follows from \eqref{4.9''} and the fact that $\text{vol}(y_1,y_1-z_1,\ldots,y_1-z_m) = \text{vol}(y_1,z_1,\ldots,z_m)$.

\smallskip

For the induction step, suppose \eqref{4.10} holds for $k$ and let $\uy=(y_1,\ldots,y_k)\in S_{\F_{\uz}}$ be point such that the vectors $y_1,\ldots,y_k,z_1,\ldots,z_m$ are linearly independent. Then the partition of variables $\uy=(\uy_1,y_k)$ with $\uy_1=(y_1,\ldots,y_{k-1})$
is admissible, thus one has
\[ 
d\om_{\F_{\uz}}(\uy) = \,d\om_{\F_{\uz}}(\uy_1)\,d\om_{\F_{\uz,\uy_1}}(y_k).
\]
The set $S_{\F_{\uz,\uy_1}}$ is a $(d-m-k)$-dimensional sphere with radius $r_k(\uy)$ given by \eqref{4.9''}, thus by \eqref{5.5} one has
\[
d\om_{\F_{\uz,\uy_1}}(y_k) = c_{d,m,T}\, \text{vol}(y_k,y_{k-1},\ldots,y_1,z_1,\ldots,z_m)^{-1}\, r_k(\uy)^{d-m-k}\,d\si^{d-m-k}(y_k).
\]
By repeated applications of \eqref{4.9} one has
\[
\text{vol}(y_k,y_{k-1},\ldots,y_1,z_1,\ldots,z_m) = \text{vol}(z_1,\ldots,z_m)\ \prod_{i=1}^k r_i(\uy).
\]
Thus by the induction hypotheses we have that 
\[
d\om_{\F_{\uz}}(\uy) = c_{d,m,T}\ 
\prod_{i=1}^{k-1} r_i (\uy_1)^{d-m-k-1}\,r_k(\uy)^{d-m-k-1}
\text{vol}(z_1,\ldots,z_m)^{-k}
\ d\si^{d-m-1}(y_1)\cdots d\si^{d-m-k}(y_k).
\]
which is \eqref{4.10} for $\uy=(y_1,\ldots,y_k)$. Inequality \eqref{4.8''} follows immediately in dimensions $d\geq m+k$ as $r_i(\uy_1)\leq |y_i| \leq C_T$ for all $1\leq i<k$.
\end{proof}


\section{Measures on real algebraic sets}\label{appendix}

Let $\F=\{f_1,\ldots,f_n\}$ be a family of polynomials $f_i:\R^d\to\R$. We will describe certain measures supported on the algebraic set 
\eq\label{5.1} S_{\F}:= \{x\in\R^d:\ f_1(x)=\ldots=f_n(x)=0\}.\ee

A point $x\in S_\F$ is called \emph{non-singular} if the gradient vectors $\nabla f_1(x),\ldots,\nabla f_n(x)$ are linearly independent, and let $S_\F^0$ denote the set of non-singular points. It is well-known  and is easy to see, that if $S_\F^0\neq\emptyset$ then it is a relative open, dense subset of $S_\F$, and moreover it is an $d-n$-dimensional sub-manifold of $\R^d$. If $x\in S_\F^0$ then there exists a set of coordinates, $J=\{j_1,\ldots,j_n\}$, with $1\leq j_1<\ldots<j_n\leq d$, such that 
\eq\label{5.2} j_{\F,J}(x):= \det\, \left( \frac{\partial f_i}{\partial x_j} (x)\right)_{1\leq i\leq n,j\in J}\ \neq 0.\ee

Accordingly, we will call a set of coordinates $J$ \emph{admissible}, if \eqref{5.2} holds for at least one point $x\in S_\F^0$, and will denote by $S_{\F,J}$ the set of such points. For a given set of coordinates $x_J$ let $\nabla_{x_J}f(x):=(\partial_{x_j}f(x))_{j\in J}$, and note that $J$ is admissible if and only if the gradient vectors $\nabla_{x_J}f_1(x),\ldots,\nabla_{x_J}f_n(x)$,  are linearly independent at at least one point $x\in S_\F$. It is clear that $S_{\F,J}$ is a relative open and dense subset of $S_\F$ and is a also $(d-n)$-dimensional sub-manifold, moreover 
\[S_\F^0 = \bigcup_{J\, \text{admissible}} S_{\F,J}.\]

We define a measure, near a point $x_0\in S_{\F,J}$ as follows. For simplicity of notation assume that $J=\{1,\ldots,n\}$ and let $\Phi(x):=(f_1,\ldots,f_n,x_{n+1},\ldots,x_d)$. Then $\Phi:U\to V$ is a diffeomorphism on some open set $x_0\in U\subs \R^d$ to its image $V=\Phi(U)$, moreover $S_\F=\Phi^{-1}(V\cap \R^{d-n})$. Indeed, $x\in S_\F\cap U$ if and only if $\Phi(x)=(0,\dots,0,x_{n+1},\ldots, x_d)\in V$. Let $I=\{n+1,\ldots,d\}$ and write $x_I:=(x_{n+1},\ldots,x_d)$. Let $\Psi(x_I)=\Phi^{-1}(0,x_I)$ and in local coordinates $x_I$ define the measure $\omega_{\F}$ via
\eq\label{5.3}\int g\,d\omega_{\F} := \int g(\Psi(x_I))\ |Jac^{-1}_\Phi (\Psi(x_I))|\,dx_I
\ee
for a continuous function $g$ supported on $U$. Note that $Jac_\Phi(x)=j_{\F,J}(x)$, i.e. the Jacobian of the mapping $\Phi$ at $x\in U$ is equal to the expression given in \eqref{5.2}, and that the measure $d\omega_\F$ is supported on $S_\F$. It is not hard to show that this measure is independent of the choice local coordinates $x_I$ and then $\omega_\F$ is defined on $S_\F^0$ as the set $S_\F^0\backslash S_{\F,J}$ is of measure zero with respect to $\omega_F$, being a proper analytic subset on $\R^{d-n}$ in any other admissible local coordinates.


A more geometric description of the measure $d\om_\F$ can be given as follows. Let $x_0\in S_\F^0$, and choose an orthonormal basis $e_1,\ldots,e_d$ of $\R^d$ such that $\text{span}\{e_1,\ldots,e_n\}=\text{span}\{\nabla f_1(x_0),\ldots,\nabla f_n (x_0)\}$ and $T_{x_0}S_\F=\text{span}\{e_{n+1},\ldots,e_d\}$. This possible as $\nabla f_i(x_0)$ is orthogonal to the tangent space of $S_\F$ at $x_0$. Let $x_1,\ldots,x_d$ be the system of coordinates associated to the basis. Then the tangent of the mapping $\Phi(x)$ at $x_0$ is a block diagonal matrix and its Jacobian $Jac_\Phi(x_0)$ is the determinant of the $n\times n$ matrix whose rows are the gradient vectors $\nabla f_i(x_0)$. Its magnitude further equals to the volume of the parallelepiped spanned by these vectors. On the other hand, the surface area measure $d\si_{S_\F}$ is equal to $dx_{n+1}\ldots dx_d$ at $x_0$ in this coordinate system, thus we have
\eq\label{47}
d\om_\F (x_0) = \text{vol}(\nabla f_1(x_0),\ldots,\nabla f_i(x_0))^{-1}\, d\si_{S_\F}(x_0).
\ee
This provides a coordinate free description of the Gelfand-Leray measure and shows that it is supported on $S_\F^0$ and is absolute continuous with respect to the surface area measure with density $\text{vol}(\nabla f_1(x_0),\ldots,\nabla f_i(x_0))^{-1}$.\\


Let $x=(z,y)$ be a partition of coordinates in $\R^d$, with $y=x_{J_2}$, $z=X_{J_1}$, and assume that for $i=1,\ldots,m$ the functions $f_i$ depend only on the $z$-variables. We say that the partition of coordinates is \emph{admissible}, if there is a point $x=(z,y)\in S_\F$ such that both the gradient vectors $\nabla_z f_1(x),\ldots,\nabla_z f_m(x)$ and the vectors $\nabla_y f_{m+1}(x),\ldots,\nabla_y f_n(x)$ for a linearly independent system. Partition the system $\F=\F_1\cup\F_2$ with $\F_1=\{f_1,\ldots,f_m\}$ and $\F_2=\{f_{m+1},\ldots,f_n\}$. Then there is set $J_1'\subs J_1$ for which 
\[j_{\F_1,J_1'}(z):= \det \left(\frac{\partial f_i}{\partial x_j}(z)\right )_{1\leq i\leq m, \,j\in J_1'}\neq 0,\]
and also a set $J_2'\subs J_2$ such that 
\[j_{\F_2,J_2'}(z,y):= \det \left(\frac{\partial f_i}{\partial x_j}(z,y)\right)_{m+1\leq i\leq n,\, j\in J_2'}\neq 0.\]
Since $\nabla_y f_i\equiv 0$ for $1\leq i\leq m$, it follows that the set of coordinates $J'=J_1'\cup J_2'$ is admissible, moreover
\[j_{\F,J'}(y,z) = j_{\F_1,J_1'}(z)\,j_{\F_2,J_2'}(y,z).\]

For fixed $z$, let $f_{i,z}(y):=f_i(z,y)$ and let $\F_{2,z}=\{f_{m+1,z},\ldots,f_{n,z}\}$. Then clearly $j_{\F_2,J_2'}(y,z)=j_{\F_{2,z},J_2'}(y)$ as it only involves partial derivatives with respect to the $y$-variables. Thus we have an analogue of Fubini's theorem, namely
\eq\label{5.4}
\int g(x)\,d\omega_\F (x) \,=\, \int\int g(z,y)\,d\omega_{\F_{2,z}}(y)\, d\omega_{\F_1}(z).\ee


Consider now algebraic sets given as the intersection of spheres. Let $x_1,\ldots,x_m\in\R^d$, $t_1,\ldots,t_m>0$ and $\F=\{f_1,\ldots,f_m\}$ where $f_i(x)=|x-x_i|^2-t_i$ for $i=1,\ldots,m$. Then $S_\F$ is the intersection of spheres centered at the points $x_i$ of radius $r_i=t_i^{1/2}$. If the set of points $X=\{x_1,\ldots,x_m\}$ is in general position (i.e they span an $m-1$-dimensional affine subspace), then a point $x\in S_\F$ is non-singular if
$x\notin \text{span}\{X\}$, i.e if $x$ cannot be written as linear combination of $x_1,\ldots,x_m$. Indeed, since $\nabla f_i(x)=2(x-x_i)$ we have that
\[\sum_{i=1}^m a_i\nabla f_i(x)=0\ \Longleftrightarrow\ \sum_{i=1}^m a_i\,x = \sum_{i=1}^m a_i x_i,\]
which implies $\sum_{i=1}^m a_i=0$ and $\sum_{i=1}^m a_i x_i=0$. By replacing the equations $|x-x_i|^2=t_i$ with $|x-x_1|^2-|x-x_i|^2=t_1-t_i$, which is of the form $x\cdot (x_1-x_i)=c_i$, for $i=2,\ldots,m$, it follows that $S_\F$ is the intersection of sphere with an $n-1$-codimensional affine subspace $Y$, perpendicular to the affine subspace spanned by the points $x_i$. Thus $S_\F$ is an $m$-codimensional sphere of $\R^d$ if $S_\F$ has one point $x\notin \text{span}\{x_1,\ldots,x_m\}$ and all of its points are non-singular. Let $x'$ be the orthogonal projection of $x$ to $\text{span}\{X\}$. If $y\in Y$ is a point with $|y-x'|=|x-x'|$ then by the Pythagorean theorem we have that $|y-x_i|=|x-x_i|$ and hence $y\in S_\F$. It follows that $S_\F$ is a sphere centered at $x'$ and contained in $Y$.

Let $T=T_X$ be the inner product matrix with entries $t_{ij}:=(x-x_i)\cdot (x-x_j)$ for $x\in S_\F$. Since $(x-x_i)\cdot (x-x_j)=1/2( t_i+t_j - |x_i-x_j|^2)$ the matrix $T$ is independent of $x$. We will show that $d\om_\F=c_T\, d\si_{S_\F}$ where $d\si_{S_\F}$ denotes the surface area measure on the sphere $S_\F$ and $c_T= 2^{-m} \det (T)^{-1/2}>0$, i.e for a function $g\in C_0(\R^d)$,
\be\label{spheres}
\int_{S_\F} g(x)\,d\om_\F(x) = c_T\int_{S_\F} g(x)\,d\si_{S_\F}(x).
\ee
Indeed by \eqref{47}, 
\eq\label{5.5}
\int_{S_\F} g(x)\,d\om_\F(x) = 2^{-m} \,\text{vol}(x-x_1,\ldots,x-x_m)^{-1}\,
\int_{S_\F} g(x)\,d\si_{S_\F}(x),
\ee
and it is a well-known fact from linear algebra that 
\[\text{vol} (x-x_1,\ldots,x-x_m)^2 = \det(T).\] 




\end{document}